\documentclass{amsart}

\usepackage{latexsym}
\usepackage{float}
\usepackage{amssymb}
\usepackage{amsmath}
\usepackage{amsthm}
\usepackage{mathrsfs}
\usepackage{euscript}
\usepackage[cmtip,all]{xy}

\newcommand{\vsp}{\vskip 0.2in}

\newlength{\whitecirclewidth}
\newlength{\whitecircleheight}
\newlength{\blackcirclewidth}
\newlength{\blackcircleheight}
\newlength{\corssandcircleheight}
\newlength{\corssandcirclewidth}
\newlength{\nodesize}
\newcommand{\whitecirclesymbol}{\makebox[\whitecirclewidth]{\Large$\circ$}}
\newcommand{\blackcirclesymbol}{\makebox[\blackcirclewidth]{\Large$\bullet$}}
\newcommand{\crossandcirclesymbol}{\makebox[\corssandcirclewidth]{\large$\otimes$}}

\newcommand{\wnode}{*-<\nodesize>{\whitecirclesymbol}}

\newcommand{\single}{\ar@{-}}
\newcommand{\rdouble}{\ar@2{->}}
\newcommand{\ldouble}{\ar@2{<-}}
\newcommand{\rtriple}{\ar@3{->}}
\newcommand{\ltriple}{\ar@3{<-}}

\newcommand{\abovewnode}[1]%
{
*-<\nodesize>{\raisebox{0pt}[\whitecircleheight][0pt]{$\overset{\rlap{$\displaystyle
#1$}}{\whitecirclesymbol}$}}}

\newcommand{\belowwnode}[1]%
{
*-<\nodesize>{\raisebox{1pt}[\whitecircleheight][0pt]{$\underset{\rlap{$\displaystyle
#1$}}{\whitecirclesymbol}$}}}

\newcommand{\abovebnode}[1]%
{
*-<\nodesize>{\raisebox{0pt}[\whitecircleheight][0pt]{$\overset{\rlap{$\displaystyle
#1$}}{\blackcirclesymbol}$}}}

\newcommand{\belowbnode}[1]%
{
*-<\nodesize>{\raisebox{0pt}[\whitecircleheight][0pt]{$\underset{\rlap{$\displaystyle
#1$}}{\blackcirclesymbol}$}}}

\newcommand{\abovecnode}[1]%
{
*-<\nodesize>{\raisebox{0pt}[\whitecircleheight][0pt]{$\overset{\rlap{$\displaystyle
#1$}}{\crossandcirclesymbol}$}}}

\newcommand{\belowcnode}[1]%
{
*-<\nodesize>{\raisebox{1pt}[\whitecircleheight][0pt]{$\underset{\rlap{$\displaystyle
#1$}}{\crossandcirclesymbol}$}}}

\newcommand{\gb}{\beta}
\newcommand{\ga}{\alpha}

\newcommand{\gl}{\lambda}

\newcommand{\gd}{\delta}
\newcommand{\gD}{\Delta}
\renewcommand{\ge}{\epsilon}
\newcommand{\gp}{\varphi}

\newcommand{\gs}{\sigma}

\newcommand{\fb}{{\mathfrak b}}
\newcommand{\fg}{{\mathfrak g}}
\newcommand{\fh}{{\mathfrak h}}

\newcommand{\fl}{{\mathfrak l}}

\newcommand{\fn}{{\mathfrak n}}

\newcommand{\fq}{{\mathfrak q}}

\newcommand{\fu}{{\mathfrak u}}

\newcommand{\fz}{{\mathfrak z}}
\newcommand{\f}{\mathfrak}

\newcommand{\eV}{\EuScript{V}}

\newcommand{\eX}{\EuScript{X}}

\newcommand{\cga}{\alpha^\vee}
\newcommand{\cgb}{\beta^\vee}

\newcommand{\nbar}{\bar{n}}                 %opposite Iwasawa

\newcommand{\R}{\mathbb{R}}          
\newcommand{\C}{\mathbb{C}}          
           
\newcommand{\Z}{\mathbb{Z}}

\newtheorem{Thm}[equation]{Theorem}
\newtheorem{Lem}[equation]{Lemma}
\newtheorem{Cor}[equation]{Corollary}
\newtheorem{Prop}[equation]{Proposition}

\newtheorem{Def}[equation]{Definition}

\numberwithin{equation}{section}
\newcommand{\be}{\begin{equation}}
\newcommand{\beu}{\begin{equation*}}

\newcommand{\acts}{ {\raisebox{1pt} {$\scriptstyle \bullet$} } }
\newcommand{\ad}{\text{ad}}
\newcommand{\Ad}{\text{Ad}}
\newcommand{\Cal}{\mathcal}

\newcommand{\Hom}{\text{Hom}}

\newcommand{\IP}[2]{\langle#1 , #2\rangle}     %inner product.

\newcommand{\flg}{\frak{l}_{\gamma}}
\newcommand{\flng}{\frak{l}_{n\gamma}}

\newcommand{\cfn}{\frak{z}(\frak{n})}
\newcommand{\xig}{\xi_\gamma}
\newcommand{\xing}{\xi_{n\gamma}}
\newcommand{\geg}{\epsilon_\gamma}
\newcommand{\geng}{\epsilon_{n\gamma}}

\newcommand{\ttau}{\tilde{\tau}}
\newcommand{\Sym}{\text{Sym}}

\newcommand{\tu}{\tilde{u}}

\newcommand{\vep}{\varepsilon}

\begin{document}

\bibliographystyle{amsplain}

\baselineskip=16pt

\title[Homomorphisms between generalized Verma modules]
{On the homomorphisms between the generalized Verma modules
arising from conformally invariant systems}                                     

\author{Toshihisa Kubo}
\address{Graduate School of Mathematical Sciences, 
The University of Tokyo,
 3-8-1 Komaba, Meguro-ku, Tokyo 153-8914, Japan}
\email{toskubo@ms.u-tokyo.ac.jp}


\maketitle

%%%%%%%%%%%%%%%%%%%%%%%%%%%%%%%%%%%%%

\begin{abstract}
It is shown by Barchini, Kable, and Zierau that 
conformally invariant systems of differential operators
yield explicit homomorphisms between certain generalized Verma modules. 
In this paper we determine whether or not the homomorphisms
arising from such systems of first and second order
differential operators  
associated to maximal parabolic subalgebras 
of quasi-Heisenberg type are standard.
\end{abstract}
%%%%%%%%%%%%%%%%%%%%%%%%%%%%%%%%%%%%%%%%%

%%%%%%%%%%%%%%%%%%%%%%%%%%%%%%%%%%%%%%%%%%
 \section{Introduction}\label{chap:intro}

%%%%%%%%%%%%%%%%%%%%%%%%%%%%%%%%%%%%%%%%%%
The main work of this paper concerns homomorphisms
between the generalized Verma modules arising from
conformally invariant systems of differential operators.
As a conformally invariant system is a central object of this paper,
we begin with introducing
the definition of such systems of operators.
Loosely speaking, a conformally invariant system is 
a system of differential operators
that are equivariant under a Lie algebra action. 
To describe the equivariance condition precisely,
let $\fg_0$ be a real Lie algebra.
The definition of conformally invariant systems requires
the notions of a $\fg_0$-manifold and $\fg_0$-bundle.
First, a smooth manifold $M$
is said to be a \emph{$\fg_0$-manifold} if there exists a Lie algebra homomorphism 
$\pi_M: \fg_0 \to C^\infty(M)\oplus \eX(M)$, 
where $\eX(M)$ is the space of smooth vector fields on $M$.
Here, the Lie algebra structure of $C^\infty(M)\oplus \eX(M)$
is the standard one induced from the algebra structure of differential operators.
Given $\fg_0$-manifold $M$, write 
$\pi_M(X) = \pi_0(X) + \pi_1(X)$ with $\pi_0(X) \in C^\infty(M)$
and $\pi_1(X) \in \eX(M)$. 
Next, let $\mathbb{D}(\eV)$ denote the space of differential operators 
on a vector bundle $\eV \to M$. 
We regard any smooth functions $f$ on $M$ 
as elements in $\mathbb{D}(\eV)$ by identifying them with
the multiplication operator that they induce.
Then we  say that a vector bundle $\eV \to M$ is
a \emph{$\fg_0$-bundle} if   
there exists a Lie algebra homomorphism 
$\pi_\eV: \fg_0 \to \mathbb{D}(\eV)$ so that 
in $\mathbb{D}(\eV)$ $[\pi_\eV(X), f]= \pi_1(X) \acts f$
for all $X \in \fg_0$ and all $f \in C^\infty(M)$,
where the dot $\acts$ denotes 
the action of the differential operator $\pi_1(X)$. 
Here, as for $C^\infty(M)\oplus \eX(M)$, the Lie algebra 
structure of $\mathbb{D}(\eV)$ is the standard one
coming from its algebra structure of operators with composition.
Now, given $\fg_0$-bundle $\eV\to M$,
a system of linearly independent differential operators
$D_1, \ldots, D_m \in \mathbb{D}(\eV)$
is called a \emph{conformally invariant system}
on $\eV$ with respect to $\pi_\eV$
if,  for all $X \in \fg_0$, it satisfies the bracket identity
\begin{equation*}
[\pi_\eV(X), D_j] = \sum_{i}^mC_{ij}^XD_i,
\end{equation*}
where $C_{ij}^X$ are smooth functions on $M$.
By extending the Lie algebra homomorphisms $\pi_M$ and $\pi_\eV$ $\C$-linearly, 
the definitions of a $\fg_0$-manifold, $\fg_0$-bundle, and 
conformally invariant system can be applied equally well to
the complexified Lie algebra $\fg = \fg_0 \otimes_\R \C$.

The Laplacian $\gD$
on $\R^n$ and wave operator $\square$ on the Minkowski space $\R^{3,1}$ 
are two typical examples for conformally invariant systems 
consisting of one differential operator.
The notion of conformally invariant systems generalizes that 
of Kostant's quasi-invariant differential operator (\cite{Kostant75}).
A systematic study of conformally invariant systems recently
started with the work of Barchini-Kable-Zierau in \cite{BKZ08} and \cite{BKZ09},
and the study of such systems of operators is continued in  
\cite{Kable11}, \cite{Kable12A}, \cite{Kable12B}, \cite{Kable12C},
\cite{Kable}, \cite{KuboThesis1}, and \cite{Kubo11}.
%The introduction of \cite{Kable} is a particularly good source 
%for the history of conformally invariant operators and their significance,
%especially their relationships with the Kelvin transform.

While the works 
\cite{BKZ08}, \cite{Kable11}-\cite{Kable}, \cite{KuboThesis1}, and \cite{Kubo11} 
mainly focus on the construction of conformally invariant systems 
or the solution spaces to such systems of operators, 
we in this paper study the homomorphisms between generalized Verma 
modules that arise from conformally invariant systems.
Homomorphisms between generalized Verma modules
(or equivalently intertwining differential operators between
degenerate principal series representations)
have received a lot of attentions from many points of views
(see for example \cite{Boe85}, \cite{Dobrev-a12}, 
\cite{Jakobsen85}, \cite{KP-a13}, and \cite{Matumoto06}).
It has been shown in \cite{BKZ09} that a conformally invariant system 
yields a homomorphism between certain generalized Verma modules,
one of which is non-scalar. 
In the present work we would like to understand the ``standardness" of 
such homomorphisms.
  A homomorphism between generalized Verma modules is 
called \emph{standard} if it is induced from a homomorphism between
the corresponding (ordinary) Verma modules, 
and called \emph{non-standard} otherwise. %(See Section \ref{SS:Hom2}.)
While standard homomorphisms are well-understood
(see for example \cite{Boe85} and \cite{Lepowsky77}), 
the classification of non-standard homomorphisms is still an open problem.
See for instance \cite{BC85}, \cite{BC86},
and Section 11.5 of \cite{BE89} for the classification of such maps
for certain cases. 
We may want to note that 
much of the published work concerning 
non-standard homomorphisms is for the case that 
the nilpotent radical $\fn$ for 
parabolic subalgebra $\fq = \fl \oplus \fn$ is abelian.

In \cite{KuboThesis1} we have built a number of 
conformally invariant systems of first and second order differential operators,
that are associated to maximal parabolic subalgebras $\fq = \fl \oplus \fn$
with nilpotent radical $\fn$ 
satisfying the conditions that $[\fn, [\fn, \fn]] = 0$ and $\dim([\fn, \fn]) >1$.
We call such nilpotent algebra $\fn$ \emph{quasi-Heisenberg} and 
such parabolic subalgebras $\fq$ 
\emph{quasi-Heinseberg type}.
%(See Subsection 2.3 for such $\fq$.)
%with  two-step nilpotent radical of non-Heisenberg type.
%that is, a maximal parabolic subalgebra $\fq= \fl \oplus \fn$  with nilpotent radical $\fn$ 
%satisfying the conditions that $[\fn, [\fn, \fn]] = 0$ and $\dim([\fn, \fn]) >1$. 
Then, in this paper, we determine whether or not the homomorphisms
between the generalized Verma modules arising 
from the systems of operators associated to maximal parabolic subalgebras
$\fq$ of quasi-Heisenberg type are standard.
As the nilpotent radical $\fn$ of $\fq = \fl \oplus \fn$ is quasi-Heisenberg, 
this gives examples of non-standard maps beyond the scope of the case
that $\fn$ is abelian.

To describe our work more precisely, we now briefly review the results
of \cite{KuboThesis1}.
Let $G$ be a complex, simple, connected, simply-connected Lie group
with Lie algebra $\fg$. Give a $\Z$-grading 
$\fg = \bigoplus_{j=-r}^r \fg(j)$ on $\fg$ so that 
$\fq = \fg(0) \oplus \bigoplus_{j>0} \fg(j) = \fl \oplus \fn$ 
is a parabolic subalgebra.
Let $Q = N_G(\fq) = LN$. For a real form $\fg_0$ of $\fg$,
define $G_0$ to be an analytic subgroup of $G$ with Lie algebra $\fg_0$.
Set $Q_0 = N_{G_0}(\fq)$. Our manifold is $M=G_0/Q_0$ and we consider
a line bundle $\Cal{L}_{s} \to G_0/Q_0$ for each $s\in \C$.
By the Bruhat theory, the homogeneous space $G_0/Q_0$ admits an open dense
submanifold $\bar{N}_0Q_0/Q_0$. We restrict our bundle to this submanifold.
By slight abuse of notation we refer to the restricted bundle as $\Cal{L}_s$.
The systems that we construct act on smooth sections of 
the restricted bundle $\Cal{L}_s$.

Our systems of operators are constructed from 
$L$-irreducible constituents $W$ of  
$\fg(-r+k) \otimes \fg(r)$ for $1\leq k \leq 2r$.
We call the systems of operators \emph{$\Omega_k$ systems}.
(We shall describe the construction more precisely in Section \ref{SS:Prelim}.)
It is not necessary that every $L$-irreducible constituent of
$\fg(-r+k) \otimes \fg(r)$ contributes to the construction for $\Omega_k$ systems.
Then we call irreducible constituents $W$ \emph{special}
if they contribute to the systems of operators.
Here, we should remark a certain discrepancy of the definition for 
special constituents between this paper and \cite{KuboThesis1}.
In \cite{KuboThesis1},
special constituents for $\Omega_2$ systems are defined 
as irreducible constituents of $\fg(0) \otimes \fg(2)$ whose
highest weights satisfy a certain technical condition.
(See Definition 6.7 of \cite{KuboThesis1}.)
In the paper we first observed that,
if an irreducible constituent of $\fg(0) \otimes \fg(2)$
contributes to an $\Omega_2$ system
then its highest weight satisfies the technical condition. 
We then tried to show that the opposite direction also holds;
namely, we tried to show that irreducible constituents with 
the highest weight condition contribute to $\Omega_2$ systems.
For all the cases but two, it is 
verified that such irreducible constituents do contribute to the construction. 
The difficulty for the two open cases is that 
there is a problem to apply to these cases
the method that is used for any other cases.
We do expect that also in the open cases the constituents with the 
highest weight condition contribute to the construction.
Thus we redefined special constituents in the way introduced
at the beginning of this paragraph, so that the definition works 
not only for $\Omega_2$ systems but also for $\Omega_k$ systems 
for general $k$. 
We would like to verify the open cases elsewhere and 
so the two definitions for special constituents do agree.

There is no reason to expect that $\Omega_k$ systems are 
conformally invariant on $\Cal{L}_{s}$ for arbitrary $s \in \C$;
the conformal invariance of $\Omega_k$ systems 
depends on the complex parameter $s$ for the line bundle $\Cal{L}_{s}$. 
We then say that an $\Omega_k$ system has \emph{special value $s_0$} 
if the system is conformally invariant on the line bundle $\Cal{L}_{s_0}$.

In \cite{KuboThesis1}, 
we found the special values of
the $\Omega_1$ system and certain $\Omega_2$ systems 
associated to a maximal parabolic subalgebra
$\fq$ of quasi-Heisenberg type.
We may want to note that,
to find the special values for $\Omega_2$ systems,
the technical condition on the highest weights for 
the special constituents plays a crucial role.
(See Section 7 of \cite{KuboThesis1}.)
See Theorem \ref{Thm8.1.2} and Table \ref{T93} for
the special values of these systems.
In Table \ref{T93}, one notices that
there are two missing cases, the cases with a question mark (?).
These are the two open cases mentioned above.
We would like to fill in the gaps in the future.

In this paper, with the special values determined in \cite{KuboThesis1} in hand,
for $k=1,2$, we classify the homomorphisms
$\varphi_{\Omega_k}$ 
between the generalized Verma modules
arising from the conformally invariant $\Omega_k$ systems 
as standard or non-standard.
Our main tool is a well-known result due to Lepowsky (Theorem \ref{Thm:Std}).
It turns out that the map $\varphi_{\Omega_k}$ is non-standard if and only if 
the special value $s_0$ of an $\Omega_k$ system is a positive integer.
See Theorem \ref{Thm:MapO1} for the result 
for the map $\varphi_{\Omega_1}$.
Table \ref{THom} summarizes
the classification for $\varphi_{\Omega_2}$.

Now we outline the rest of this paper. 
This paper consists of six sections with this introduction and one appendix.
In Section 2 we recall from \cite{KuboThesis1} the construction of 
the $\Omega_k$ systems. We also review maximal
parabolic subalgebras $\fq$ of quasi-Heisenberg type in this section.
Section 3 discusses the relationship between conformally invariant 
$\Omega_k$ systems and homomorphisms between generalized Verma modules.
We start Section 4 with reviewing the general facts on the standard homomorphisms.
We then specialize such facts to the situation that we concern. 

In Sections 5 and 6, for $k=1,2$, 
we determine whether or not the homomorphisms
$\varphi_{\Omega_k}$ arising from the $\Omega_k$ systems 
associated to the maximal
parabolic subalgebra $\fq$ under consideration
are standard.
This is done in four theorems, namely, 
Theorem \ref{Thm:MapO1}, 
Theorem \ref{Thm:MapO2Type2},
Theorem \ref{Thm:MapO2Int}, and 
Theorem \ref{Thm:MapO2Bn(i)}.

Finally, in Appendix \ref{chap:Data}, we recall from \cite{KuboThesis1}
the miscellaneous useful data for the parabolic subalgebras 
under consideration. The data will be referred to in several proofs
in this paper.

%%%%%%%%%%%%%%%%%%%%%%%%%%%%%%%%%%%%%%%%%%
\vskip 0.1in
\noindent \textbf{Acknowledgment.}
This work is part of author's Ph.D. thesis at Oklahoma State University.
The author would like to thank his advisor, Leticia Barchini, 
for her generous guidance. 
He would also like to thank Anthony Kable and  
Roger Zierau for their valuable comments on this work,
and the referee for careful reading.

%%%%%%%%%%%%%%%%%%%%%%%%%%%%%%%%%%%%%%%%%
\section{Preliminaries}\label{SS:Prelim}

The purpose of this section is to recall from \cite{KuboThesis1} 
our construction of systems of differential operators.
We also review the maximal parabolic subalgebras of quasi-Heisenberg type.
%The generalized Verma modules that we concern in this paper
%are arising from the systems of operators associated to the maximal
%parabolic subalgebras.

%%%%%%%%%%%%%%%%%%%%%%%%%%%%%%%%%%%%%%%%%
\subsection{A specialization of the theory}\label{SS:Setup}

First we recall from Subsection 2.1 in \cite{KuboThesis1} 
the $\fg$-manifold and $\fg$-bundle that we study in this paper.
Let $G$ be a complex, simple, connected, simply-connected Lie group
with Lie algebra $\fg$. Such $G$ contains a maximal connected solvable
subgroup $B$. Write $\fb =\fh \oplus \fu$ for its Lie algebra with $\fh$ 
the Cartan subalgebra and $\fu$ the nilpotent subalgebra. 
Let $\fq \supset \fb$ be a parabolic subalgebra of $\fg$. 
We define $Q = N_G(\fq)$, a parabolic subgroup of $G$.
%It follows from Section 8.4 of \cite{Springer08} that $Q$ is connected. 
Write $Q = LN$ for the Levi decomposition of $Q$.
%with $L$ the Levi subgroup and $N$ the nilpotent subgroup.
%It is known, see Corollary 7.11 of \cite{Knapp02}, that 
%the Levi subgroup $L$ is the commuting product
%$L = Z(L)^\circ L_{ss}$,
%where $Z(L)^\circ$ is the identity component of the center of $L$ and 
%$L_{ss}$ is the semisimple part of $L$.

Let $\fg_0$ be a real form of $\fg$ in which 
the complex parabolic subalgebra $\fq$ has a real form $\fq_0$,
and let $G_0$ be the analytic subgroup of $G$ with Lie algebra $\fg_0$. 
Define $Q_0 = N_{G_0}(\fq) \subset Q$, and write $Q_0=L_0N_0$.
We will work with $G_0/Q_0$ for a class of maximal parabolic subgroup $Q_0$
whose Lie algebra $\fq_0$ is of two-step nilpotent type.

Next, let $\gD = \gD(\fg,\fh)$ be the set of roots of $\fg$ with respect to $\fh$.
Let $\gD^+$ be the positive system attached to $\fb$ and 
denote by $\Pi$ the set of simple roots.
We write $\fg_\ga$ for the root space for $\ga \in \gD$.
For each subset $S \subset \Pi$, let $\fq_S$ be the 
corresponding standard parabolic subalgebra. 
Write $\fq_S = \fl_S \oplus \fn_S$ with Levi factor 
$\fl_S = \fh \oplus  \bigoplus_{\ga \in \gD_S}\fg_\ga$ and 
nilpotent radical $\fn_S = \bigoplus_{\ga \in \gD^+ \backslash \gD_S} \fg_\ga$,
where $\gD_S = \{ \ga \in \gD \; | \; \ga \in \textrm{span}(\Pi \backslash S) \}$.
If $Q_0$ is a maximal parabolic subgroup then 
there exists a unique simple root $\ga_\fq \in \Pi$ so that 
$\fq = \fq_{\{\ga_\fq\}}$. Let $\gl_\fq$ be the fundamental weight of $\ga_\fq$.
The weight $\gl_\fq$ is orthogonal to any roots $\ga$ 
with $\fg_\ga \subset [\fl, \fl]$. Hence it exponentiates to 
a character $\chi_\fq$ of $L$. As $\chi_\fq$ takes real values
on $L_0$, for $s \in \C$, character $\chi^{s} =|\chi_\fq|^{s}$
is well-defined on $L_0$. 
Let $\C_{\chi^{s}}$ be the one-dimensional representation of $L_0$
with character $\chi^{s}$. 
The representation $\chi^{s}$ is extended to a representation of $Q_0$
by making it trivial on $N_0$. It then deduces a line bundle $\Cal{L}_{s}$
on $ G_0/Q_0$ with fiber $\C_{\chi^{s}}$.

The group $G_0$ acts on the space
\begin{equation*}
C^\infty_{\chi}(G_0/Q_0, \C_{\chi^{s}})
= \{ F \in C^\infty(G_0, \C_{\chi^{s}}) \; |\;
\text{$F(gq) = \chi^{s}(q^{-1})F(g)$ for all $q \in Q_0$ and $g \in G_0$} \}
\end{equation*}
\noindent by left translation.
The action $\pi_s$ of $\fg_0$ on $C^\infty_\chi(G_0/Q_0, \C_{\chi^{s}})$
arising from this action is given by
\begin{equation}\label{Eqn:Action}
(\pi_s(Y) \acts F)(g) = \frac{d}{dt}F(\exp(-tY)g)\big|_{t=0}
\end{equation}

\noindent for $Y \in \fg_0$.
This action is extended $\C$-linearly
to $\fg$ and then naturally to the universal enveloping algebra $\Cal{U}(\fg)$.
We use the same symbols for the extended actions.

Let $\bar{N}_0$ be the unipotent subgroup opposite to $N_0$.
The natural infinitesimal action of $\fg$ on the image of 
the restriction map 
$C^\infty_{\chi}(G_0/Q_0, \C_{\chi^{s}}) \to C^\infty(\bar N_0, \C_{\chi^{s}})$
induced by (\ref{Eqn:Action}) gives an action of $\fg$ on the whole space
$C^\infty(\bar N_0, \C_{\chi^{s}})$. We also denote by $\pi_s$ the induced action.
Observe that we have the direct sum $\fg = \bar \fn \oplus \fq$.
If we write $Y = Y_{\bar \fn} + Y_{\fq}$ for the decomposition 
of $Y \in \fg$ in this direct sum then, 
for $Y \in \fg$ and $f \in C^\infty(\bar N_0, \C_{\chi^{s}})$,
the derived action of $\fg$ on 
$C^\infty(\bar N_0, \C_{\chi^{s}})$
is given by
\begin{equation}\label{Eqn7.2.2}
\big(\pi_s(Y) \acts f \big)(\nbar) = s\gl_\fq \big( (\Ad(\nbar^{-1})Y)_\fq \big) f(\nbar) 
-\big( R \big( (\Ad(\nbar^{-1})Y)_{\bar \fn} \big) \acts f \big)(\nbar),
\end{equation}
where $R$ is the infinitesimal right translation of $\fg$.
The line bundle $\Cal{L}_{s} \to G_0/Q_0$ restricted to
$\bar{N}_0$ is the trivial bundle 
$\bar{N}_0 \times \C_{\chi^{s}} \to \bar{N}_0$.
By slight abuse of notation, we refer to the trivial bundle over $\bar{N}_0$
as $\Cal{L}_{s}$. It follows from the observation in Subsection 2.1 
in \cite{KuboThesis1} that $\bar{N}_0$ and $\Cal{L}_{s} \to \bar{N}_0$
are a $\fg$-manifold and $\fg$-bundle, respectively.

%%%%%%%%%%%%%%%%%%%%%%%%%%%%%%%%%%%%%%%%%
\subsection{The $\Omega_k$ systems}\label{SS:Setup}

In this subsection we briefly recall from Subsection 3.1 of \cite{KuboThesis1}
our construction of differential operators.
For a subspace $W$ of $\fg$, we write 
$\gD(W) = \{\ga \in \gD \; | \; \fg_\ga \subset W\}$
and $\Pi(W) = \gD(W) \cap \Pi$.
We keep the notation from the previous subsection, unless otherwise specified.

Let $\fg = \bigoplus_{j=-r}^r \fg(j)$ be a 
$\Z$-grading on $\fg$ with $\fg(1) \neq 0$.
For $1\leq k \leq 2r$, we define a map 
$\tau_k: \fg(1) \to \fg(-r+k) \otimes \fg(r)$  
by $X \mapsto \frac{1}{k!} \big( \ad(X)^k\otimes \text{Id}\big) \omega$
with $\omega = \sum_{\gamma_j \in \gD(\fg(r))}
X_{-\gamma_j} \otimes X_{\gamma_j}$,
where $X_{\gamma_j}$ are root vectors for $\gamma_j$
so that $\{X_{\gamma_j}, X_{-\gamma_j}, [X_{\gamma_j}, X_{-\gamma_j}]\}$
is an $\f{sl}(2)$-triple.
%Here, we mean by $\ad(X)^k \omega$ that 
%$X$ acts on the tensor product diagonally via the action $\ad(\cdot)$
%$k$ times. 
Take $L$ to be the analytic subgroup 
of $G$ with Lie algebra $\fg(0)$, and
let $W$ be an $L$-irreducible constituent of 
$\fg(-r+k) \otimes \fg(r)$.
Write $\Cal{P}^k(\fg(1))$ for the space of polynomials on $\fg(1)$
of homogeneous degree $k$.
If $W^*$ is the dual space of $W$ with respect to the Killing form $\kappa$
then there exists an $L$-intertwining operator 
$\ttau_k|_{W^*} \in \Hom_L(W^*, \Cal{P}^k(\fg(1)))$
so that, for $Y^* \in W^*$,
$\ttau_k|_{W^*}(Y^*)(X) = Y^*(\tau_k(X))$.
Here, we may want to note that $Y^*(\tau_k(X))$ is well-defined 
for $\tau_k(X) \notin W$. Indeed, observe that,
as $\fg(-r+k)^* \otimes \fg(r)^* \cong \fg(r-k) \otimes \fg(-r)$
via the Killing form $\kappa$,
the element $Y^* \in W^*\subset \fg(-r+k)^* \otimes \fg(r)^*$ 
is a linear combination of 
$\kappa(X_{\ga},  \cdot)\otimes \kappa(X_{\gb}, \cdot)$
with constant coefficients,
where $X_{\ga}$ and $X_{\gb}$ are root vectors for 
$\ga\in \gD(\fg(r-k))$ and $\gb \in \gD(\fg(-r))$.
If $Y^* = \sum_{\ga, \gb} c_{\ga, \gb}\; \kappa(X_{\ga}, \cdot)\otimes \kappa(X_{\gb}, \cdot)$ 
with constants $c_{\ga, \gb}$ then
$Y^*(\tau_k(X))$ is given by
$Y^*(\tau_k(X)) = (1/k!)\sum_{\gamma_j, \ga, \gb}c_{\ga,\gb}\;
\kappa(X_{\ga}, \ad(X)^k(X_{-\gamma_j}))\;\kappa(X_{\gb}, X_{\gamma_j})$.

 If $\ttau_k|_{W^*} \not \equiv 0$ then we call
the irreducible constituent $W$
\textbf{special} for $\tau_k$.
%\footnote{There is some discrepancy 
%on the definition for the special constituents between this paper 
%and \cite{KuboThesis1}. See the remarks about it
%in the introduction of this paper.}
Given special constituent $W$ for $\tau_k$,  
we consider the following composition of linear maps:
\begin{equation}\label{Eqn:Comp}
W^* \stackrel{\ttau_k|_{W^*}}{\to} \Cal{P}^k(\fg(1))
\cong \Sym^k(\fg(-1)) \stackrel{\gs}{\hookrightarrow} \Cal{U}(\bar \fn)
\stackrel{R}{\to} \mathbb{D}(\Cal{L}_{s})^{\bar \fn}.
\end{equation}
Here, $\gs: \Sym^k(\fg(-1)) \to \Cal{U}(\bar \fn)$ is the symmetrization operator
and $\mathbb{D}(\Cal{L}_{s})^{\bar \fn}$ is the space of $\bar \fn$-invariant 
differential operators for $\Cal{L}_{s}$. 
Let $\Omega_k|_{W^*}: W^* \to \mathbb{D}(\Cal{L}_{s})^{\bar \fn}$ be 
the composition of linear maps, namely,
$\Omega_k|_{W^*} = R \circ \gs \circ \ttau_k|_{W^*}$.
%\begin{equation*}
%\Omega_k|_{W^*} = R \circ \gs \circ \ttau_k|_{W^*}.
%\end{equation*}
For simplicity we write 
$\Omega_k(Y^*)=\Omega_k|_{W^*}(Y^*)$  
for the differential operator arising from $Y^* \in W^*$.
Note that the linear operator 
$\Omega_k|_{W^*}: W^* \to \mathbb{D}(\Cal{L}_{s})^{\bar \fn}$
is an $L_0$-intertwining operator.
%with respect to the $L$-action
%given in (2.11) in \cite{KuboThesis1}. 
(See the observation at the end of Section 3.1 of \cite{KuboThesis1}.)

Now, given basis $\{Y^*_1, \ldots, Y^*_m\}$ for $W^*$,
we have a system of differential operators
\begin{equation}
\Omega_k(Y^*_{1}), \ldots, \Omega_k(Y^*_{m}).
\end{equation}
We call such a system of operators the \textbf{$\Omega_k|_{W^*}$ system}.
When the irreducible constituent $W^*$ is not important, 
we simply refer to each $\Omega_k|_{W^*}$ system
as an $\Omega_k$ system. 
We may want to note that 
$\Omega_k$ systems are independent of the choice for a basis for $W^*$
up to some natural equivalence. (See Definition 3.5 of \cite{KuboThesis1}.)

It is important to notice that 
it is not necessary for $\Omega_k$ systems to be conformally 
invariant; their conformal invariance strongly depends on the complex
parameter $s$ for the line bundle $\Cal{L}_{s}$. 
So we say that an $\Omega_k$ system has 
\textbf{special value $s_0$} if
the system is conformally invariant on the line bundle $\Cal{L}_{s_0}$.
In \cite{KuboThesis1}, we have found the special values for the 
$\Omega_1$ system and certain $\Omega_2$ systems
associated to maximal parabolic subalgebras $\fq$ of quasi-Heinseberg type.
We shall show the special values in Sections \ref{SS:Omega1Hom}
and \ref{SS:Omega2Hom}, respectively.

%%%%%%%%%%%%%%%%%%%%%%%%%%%%%%%%%%%%%%%%% 
\subsection{Maximal parabolic subalgebras of quasi-Heisenberg type}
\label{SS:TwoStep}

In Sections \ref{SS:Omega1Hom} and \ref{SS:Omega2Hom}, 
with the special values determined in \cite{KuboThesis1} in hand,
we shall determine whether or not the homomorphisms
%$\varphi_{\Omega_1}$ and $\varphi_{\Omega_2}$ 
arising from the $\Omega_1$ system and $\Omega_2$ systems
associated to maximal parabolic subalgebras $\fq$ of 
quasi-Heisenberg type are standard.
Then, in this section, we recall from Section 4 of \cite{KuboThesis1} 
such maximal parabolic subalgebras $\fq$.

First, we call a maximal parabolic subalgebra $\fq = \fl \oplus \fn$
\textbf{quasi-Heisenberg type} 
if its nilradical $\fn$ satisfies the conditions that 
$[\fn, [\fn, \fn]] = 0$ and $\dim([\fn, \fn]) >1$.
Let $\ga_\fq$ be a simple root, so that 
the parabolic subalgebra $\fq=\fq_{\{\ga_\fq\}} = \fl \oplus \fn$ 
determined by $\ga_\fq$ 
is of quasi-Heisenberg type.
Let $\IP{\cdot}{\cdot}$ be the inner product induced on $\fh^*$
corresponding to the Killing form $\kappa$.
Write $||\ga||^2 = \IP{\ga}{\ga}$ for $\ga \in \gD$.
The coroot of $\ga$ is $\ga^{\vee} = 2\ga/\IP{\ga}{\ga}$.
%If $V$ is an $\ad(\fh)$-invariant subspace of $\fg$ with $V \neq \fg$
%then we denote by $\gD(V)$ the set of roots $\ga$ so that $\fg_\ga \subset V$.
%We write $\gD^+(V) = \gD^+ \cap \gD(V)$.

Recall from Subsection 2.1
that $\gl_\fq$ denotes the fundamental weight for $\ga_\fq$.
%As $\gD(\fl) = \{\ga \in \gD \; | \; \ga \in \text{span}(\Pi \backslash\{\ga_\fq\})\}$
%and $\gD(\fn) = \gD^+ \backslash \gD(\fl)$, we have
%$\IP{\gl_\fq}{\gb} =0$ if $\gb \in \gD(\fl)$ and 
%$\IP{\gl_\fq}{\gb} >0$ if $\gb \in \gD(\fn)$.
%\beu
%\IP{\gl_\fq}{\gb} 
%\begin{cases}
%= 0 &\text{ if $\gb \in \gD(\fl)$ }\\
%> 0 & \text{ if $\gb \in \gD(\fn)$ }.
%\end{cases}
%\end{equation*}
If $H_{\gl_\fq} \in \fh$ is defined by $\kappa(H, H_{\gl_\fq}) = \gl_\fq(H)$ 
for all $H \in \fh$ and if
%\begin{equation}\label{Eqn2.1.14}
$H_\fq = (2/||\ga_\fq||^2)H_{\gl_\fq}$
%\end{equation}
\noindent then as $\fq$ has two-step nilpotent radical,
for $\gb \in \gD^+$, $\gb(H_\fq)$ can only take the values of 
$0$, $1$, or $2$.
Therefore,
if $\fg(j)$ denotes the $j$-eigenspace of $\ad(H_\fq)$
then the action of $\ad(H_\fq)$ on $\fg$ induces a 2-grading
$\fg = \bigoplus_{j=-2}^2\fg(j)$
with parabolic subalgebra
$\fq = \fg(0) \oplus \fg(1) \oplus \fg(2)$,
where $\fl = \fg(0)$ and $\fn = \fg(1) \oplus \fg(2)$.
The subalgebra $\bar \fn$, 
the nilpotent radical opposite to $\fn$, is given by
$\bar \fn = \fg(-1) \oplus \fg(-2)$.
Here we have $\fg(0) = \fl$, $\fg(2) = \fz(\fn)$ and $\fg(-2) = \fz(\bar\fn)$,
where $\fz(\fn)$ (resp. $\fz(\bar \fn)$) is the center of $\fn$ (resp. $\bar \fn$).
Thus we denote the 2-grading  on $\fg$ by
\begin{equation}\label{Eqn4.1.6}
\fg = \fz(\bar \fn) \oplus \fg(-1) \oplus \fl \oplus \fg(1) \oplus \cfn
\end{equation}
\noindent with parabolic subalgebra 
\begin{equation*}\label{Eqn4.1.7}
\fq = \fl \oplus \fg(1) \oplus \fz(\fn).
\end{equation*}
Therefore the maps $\tau_k$ 
associated to the grading (\ref{Eqn4.1.6})
are given by
\begin{equation}\label{Eqn:TauTwoStep}
\tau_k: \fg(1) \to \fg(-2+k) \otimes \fz(\fn)
%X &\mapsto \frac{1}{k!} \ad(X)^k \omega \nonumber
\end{equation}
for $1 \leq k \leq 4$.

We next consider the structure of the 
Levi subalgebra $\fl = \fz(\fl) \oplus [\fl,\fl]$,
where $\fz(\fl)$ is the center of $\fl$.
Observe that $\fz(\fl)$ is one-dimensional.
Indeed, we have
$\fz(\fl) = \bigcap_{\ga \in \Pi(\fl)}\ker(\ga)$
with $\Pi(\fl) = \Pi \backslash \{\ga_\fq\}$.
As $\fl = \fg(0)$, we have 
$\ga(H_\fq) = 0$ for all $\ga \in \gD(\fl)$.
Thus, $H_\fq$ is an element of $\fz(\fl)$,
and so  we have $\fz(\fl) = \C H_\fq$. 

To observe the semisimple part $[\fl,\fl]$ of $\fl$,
let $\gamma$ be the highest root of $\fg$.
If $\fg$ is not of type $A_n$ then there is exactly one simple root
that is not orthogonal to $\gamma$.
Let $\ga_\gamma$ be the unique simple root so that
$\fq' = \fq_{\{\ga_\gamma\}}$ is the  
parabolic subalgebra of Heisenberg type; that is, 
its nilradical $\fn'$ satisfies 
$\dim([\fn', \fn']) = 1$.
Hence, if $\fq = \fq_{\{\ga_\fq\}}$ is a parabolic subalgebra of 
quasi-Heisenberg type 
then $\ga_\gamma$ is in $\Pi(\fl) = \Pi \backslash \{\ga_\fq\}$.
The semisimple part $[\fl,\fl]$ is either simple or the direct sum of two or three 
simple ideals with only one simple ideal containing the root space 
$\fg_{\ga_\gamma}$ for $\ga_\gamma$.
Given Dynkin type $\Cal{T}$ of $\fg$,
if we write $\Cal{T}(i)$ for the Lie algebra together with the
choice of maximal parabolic subalgebra $\fq = \fq_{\{\ga_i\}}$
determined by $\ga_i$ 
then the three simple factors occur only when
$\fq$ is of type $D_n(n-2)$.
So, if $\fq$ is not of type $D_n(n-2)$ then 
there are at most two simple factors. In this case
we denote by $\flg$ (resp. $\flng$) the simple ideal of $\fl$
that contains (resp. does not contain) $\fg_{\ga_\gamma}$.
Thus $\fl$ may decompose into
\begin{equation}\label{Eqn4.1.8}
\fl = \C H_\fq \oplus \flg \oplus \flng.
\end{equation}
Note that when $[\fl, \fl]$ is a simple ideal, we have $\flng = \{0\}$.
(See Appendix \ref{chap:Data}.) 
The maximal parabolic subalgebras $\fq=\fl \oplus \fn$ 
of quasi-Heisenberg type with the decomposition 
(\ref{Eqn4.1.8}) are given as follows:

\begin{equation} \label{Eqn4.0.1}
B_n(i)\; (3 \leq i \leq n), \quad C_n(i) \; (2 \leq i \leq n-1), \quad D_n(i)\; (3\leq i \leq n-3),
\end{equation}
\noindent and 
\begin{equation}\label{Eqn4.0.2}
E_6(3),\; E_6(5),\; E_7(2),\; E_7(6),\; E_8(1), \;F_4(4).
\end{equation}
Here, the Bourbaki conventions \cite{Bourbaki08} are used for the labels of the simple roots.
Note that, in type $A_n$,  
any maximal parabolic subalgebra has abelian nilpotent radical,
and also that, in type $G_2$, 
the two maximal parabolic subalgebras are of
either 3-step nilpotent type or Heisenberg type.

%%%%%%%%%%%%%%%%%%%%%%%%%%%%%%%%%%%%%%%%%
\section{The $\Omega_k$ systems and generalized Verma modules}
\label{SS24}

The aim of this section is to show that conformally invariant 
$\Omega_k$ systems induce non-zero $\Cal{U}(\fg)$-homomorphisms 
between certain generalized Verma modules.
The main idea is that conformally invariant $\Omega_k$ systems yield
finite dimensional simple $\fl$-submodules of generalized Verma modules, 
on which $\fn$ acts trivially.

In general, to describe the relationship between conformally invariant systems
and generalized Verma modules, we realize generalized Verma modules
as the space of smooth distributions supported at the identity. 
However, in our setting that the $\fg$-bundle is a line bundle $\Cal{L}_{s}$,
it is not necessary to use such a realization. 
Thus, in this paper, we are going to describe the relationship 
without using the realization.
For the general theory see Sections 3, 5, and 6 of \cite{BKZ09}.

A generalized Verma module 
$M_\fq[W] := \Cal{U}(\fg) \otimes_{\Cal{U}(\fq)}W$ 
is a $\Cal{U}(\fg)$-module
that is induced from a finite dimensional simple $\fl$-module $W$ 
on which $\fn$ acts trivially.
Observe that if
$\C_{-s\gl_\fq}$ is the $\fq$-module derived from 
the $Q_0$-representation $(\chi^{-s}, \C)$ then 
the differential operators in 
$\mathbb{D}(\Cal{L}_{s})^{\bar \fn}$ can be described in terms of 
elements of $M_\fq[\C_{-s\gl_\fq}]$. Indeed, 
by identifying $M_\fq[\C_{-s\gl_\fq}] $ as $\Cal{U}(\bar\fn) \otimes \C_{-s\gl_\fq}$, 
the map $M_\fq[\C_{-s\gl_\fq}]  \to \Cal{U}(\bar \fn)$ given by 
$u \otimes 1 \mapsto u$ is an isomorphism of vector spaces. 
Then the composition
\begin{equation}\label{Eqn271}
M_\fq[\C_{-s\gl_\fq}]  \to \Cal{U}(\bar \fn) \stackrel{R}{\to}
\mathbb{D}(\Cal{L}_{s})^{\bar \fn}
\end{equation}
is a vector-space isomorphism.
%We define an action of $L$ 
%on $M_\fq[\C_{-s\gl_\fq}] $ by $l \cdot (u \otimes z) = \Ad(l)u \otimes z$.
%With this action and the action (\ref{Eqn234}),
%the isomorphism (\ref{Eqn271}) is $L_0$-equivariant.

Define
\begin{equation*}
M_\fq[W]^{\fn} = \{ v \in M_\fq[W] \; | \; 
X \cdot v = 0 \; \text{for all}\; X \in \fn\}.
\end{equation*}
The following result is the specialization of Theorem 19 in \cite{BKZ09}
to the present situation. 
For the definitions for straight, homogeneous, $L_0$-stable  
conformally invariant systems, see p. 797, p. 804 and p. 806 of \cite{BKZ09}.

\begin{Thm}\label{Thm242}
If $D=D_1, \ldots, D_m$ is a straight, homogeneous, $L_0$-stable  
conformally invariant system on the line bundle $\Cal{L}_{s}$,
and if $\omega_j$ denotes the element in $\Cal{U}(\bar \fn)$
that corresponds to $D_j$ for $j=1,\ldots, m$
via right differentiation $R$ in (\ref{Eqn271}) then the space
\begin{equation*}
F(D) = \text{\emph{span}}_\C\{\omega_j \otimes 1 \; | \; j = 1, \ldots, m\}
\end{equation*}
is an $L$-submodule of $M_\fq[\C_{-s\gl_\fq}]^{\fn}$.
\end{Thm}

Now, let $W$ be a special constituent of $\fg(-r+k) \otimes \fg(k)$ for $\tau_k$.
%so that the $L_0$-intertwining operator
%$\Omega_k|_{W^*}:W^* \to \mathbb{D}(\Cal{L}_{s})^{\bar \fn}$
%is not identically zero.
Let $\omega_k|_{W^*}: W^* \to \Cal{U}(\bar \fn)$ be 
the linear operator so that 
$\omega_k|_{W^*}(Y^*)$ 
is the element in $\gs(\Sym^k(\fn)) \subset \Cal{U}(\bar \fn)$
that corresponds to the differential operator
$\Omega_k(Y^*) =\Omega_k|_{W^*}(Y^*)$
in $\mathbb{D}(\Cal{L}_{s})^{\bar \fn}$,
via right differentiation $R$ in (\ref{Eqn:Comp}).
As for $\Omega_k(Y^*)$, for simplicity, we write
$\omega_k(Y^*) = \omega_k|_{W^*}(Y^*)$. 
Then, given basis $\{Y^*_1, \ldots, Y^*_m\}$ for $W^*$,
%$\Omega_k|_{W^*}$ system
%\begin{equation*}
%$\Omega_k(Y^*_1), \ldots, \Omega_k(Y^*_m)$
%\end{equation*}
the space $F(\Omega_k|_{W^*})$ for the $\Omega_k|_{W^*}$ system
$\Omega_k|_{W^*} = \Omega_k(Y^*_1), \ldots, \Omega_k(Y^*_m)$ 
is given by
\begin{equation}\label{Eqn:Ltype1}
F(\Omega_k|_{W^*}) = \text{span}_\C\{
\omega_k(Y^*_j) \otimes 1 \; | \; j = 1, \ldots, m\}
\subset M_\fq[\C_{-s\gl_\fq}].
\end{equation}

\begin{Prop}\label{Cor243}
Suppose that special constituent $W^*$ has highest weight $\nu$.
\begin{enumerate}
\item The space $F(\Omega_k|_{W^*})$ is the simple $L$-submodule of 
$M_\fq[\C_{-s\gl_\fq}]$ with highest weight $\nu-s\gl_\fq$.
\item Moreover, if the $\Omega_k|_{W^*}$ system is conformally invariant on 
the line bundle $\Cal{L}_{s_0}$ then 
$F(\Omega_k|_{W^*})$ is a simple $L$-submodule of 
$M_\fq[\C_{-s_0\gl_\fq}]^{\fn}$ with highest weight $\nu-s_0\gl_\fq$. 
\end{enumerate}
\end{Prop}

\begin{proof}
First observe that,
by the $L_0$-equivariance of the operator $\Omega_k|_{W^*}:W^* \to \mathbb{D}(\Cal{L}_s)$,
for $l \in L$ and $Y^* \in W^*$, we have
\begin{equation*}
\omega_k(l \cdot Y^*) = \Ad(l)\omega_k(Y^*),
\end{equation*}
where the action $l \cdot Y^*$ is the standard action of $L$ on $W^*$,
which is induced from the adjoint action of $L$ on $W$.
This shows the $L$-invariance of $F(\Omega_k|_{W^*})$.
To show the irreducibility
observe that there exists a vector space isomorphism
\begin{equation*}\label{EqnF82}
F(\Omega_k|_{W^*}) \to W^* \otimes \C_{-s\gl_\fq},
\end{equation*}
that is given by $\omega_k(Y^*_j) \otimes 1 
\mapsto Y^*_j \otimes 1$.
It is clear that this vector space isomorphism 
is $L$-equivariant with respect to the standard action of $L$ on
the tensor products 
$F(\Omega_k|_{W^*})\subset \Cal{U}(\bar \fn) \otimes \C_{-s\gl_\fq}$ and
$W^* \otimes \C_{-s\gl_\fq}$.
In particular, if $W^*$ has highest weight $\nu$ then
$F(\Omega_k|_{W^*})$ is the simple $L$-module
with highest weight $\nu-s\gl_\fq$.

Note that, 
by Remark 3.8 in \cite{KuboThesis1},
if the $\Omega_k|_{W^*}$ system 
is conformally invariant then it is a straight, $L_0$-stable, and 
homogeneous system. Now the second assertion follows from
the first and Theorem \ref{Thm242}. 
\end{proof}

Now, if the $\Omega_k|_{W^*}$ system is conformally invariant 
on $\Cal{L}_{s_0}$ then, by Proposition \ref{Cor243},
$F(\Omega_k|_{W^*})$
is a simple $\fl$-submodule of 
$M_\fq[\C_{-s_0\gl_\fq}])$ on which $\fn$ acts trivially.
Thus the inclusion map
$\iota \in \Hom_{L} \big( F(\Omega_k|_{W^*}), M_\fq[\C_{-s_0\gl_\fq}] \big)$
induces a non-zero $\Cal{U}(\fg)$-homomorphism 
\begin{equation*}
\varphi_{\Omega_k} \in \Hom_{\Cal{U}(\fg), L}
\big(M_\fq[F(\Omega_k|_{W^*})], M_\fq[\C_{-s_0\gl_\fq}] \big)
\end{equation*}
between the generalized Verma modules, that is given by
\begin{align}\label{Eqn823}
M_\fq[F(\Omega_k|_{W^*})] &\stackrel{\varphi_{\Omega_k}}{\to} M_\fq[\C_{-s_0\gl_\fq}]\\
u \otimes \big( \omega_k(Y) \otimes 1) &\mapsto 
u \cdot \iota\big( \omega_k(Y) \otimes 1) \nonumber.
\end{align}

If $F(\Omega_k|_{W^*}) =\C_{-s_0\gl_\fq}$
then the map in (\ref{Eqn823}) is just the identity map. 
However, Proposition \ref{Prop244} below shows that 
it does not happen.

\begin{Prop}\label{Prop244}
If the $\Omega_k|_{W^*}$ system is conformally invariant on 
the line bundle $\Cal{L}_{s_0}$ then $F(\Omega_k|_{W^*}) \neq \C_{-s_0\gl_\fq}$.
\end{Prop}

\begin{proof}
Observe that if $\nu$ is the highest weight for $W^*$ then
$F(\Omega_k|_{W^*})$ has highest weight $\nu-s_0\gl_\fq$.
If $F(\Omega_k|_{W^*}) = \C_{-s_0\gl_\fq}$ then $\nu = 0$, and so
the irreducible constituent $W \subset \fg(-r+k)\otimes \fg(r)$
would also have highest weight 0.
As $\gamma$ is the highest weight for $\fg(r)$,
the highest weight of any irreducible constituent
of $\fg(-r+k)\otimes \fg(r)$ is of the form $\gamma + \eta$ with $\eta$  some
weight for $\fg(-r+k)$. %(see for instance \cite[Proposition 3.2]{Kumar10}). 
Thus, the highest weight 0 for $W$ 
must be of the form $0=\gamma + (-\gamma)$. 
However, since only $\fg(-r)$ has weight $-\gamma$,
it cannot be a weight for $\fg(-r+k)$ unless $k=0$.
As $k = 1, \ldots, 2r$ (see Subsection 2.2), this shows that 
$F(\Omega_k|_{W^*}) \neq \C_{-s_0\gl_\fq}$.
\end{proof}

\begin{Cor}\label{CorGVM}
If the $\Omega_k|_{W^*}$ system is conformally invariant on 
the line bundle $\Cal{L}_{s_0}$ then
the generalized Verma module $M_\fq[\C_{-s_0\gl_\fq}]$ is reducible.
\end{Cor}

\begin{proof}
This immediately follows from (\ref{Eqn823}) and Proposition \ref{Prop244}.
\end{proof}

%By Corollary \ref{CorGVM}, if an $\Omega_2$ system is conformally invariant
%over the line bundle $\Cal{L}(s_0\gl_\fq)$ then the generalized Verma module
%$M_\fq[\C_{-s_0\gl_\fq}] = \Cal{U}(\fg)\otimes_{\Cal{U}(\fq)}\C_{-s_0\gl_\fq}$
%is reducible. Table \ref{T93-2} summarizes 
%the generalized Verma modules 
%that correspond to the line bundles in Table \ref{T93}.

%\begin{table}[h]
%\caption{The Generalized Verma Modules 
%corresponding to $\Cal{L}(s_0\gl_\fq)$ in Table \ref{T93}}
%\begin{center}
%\begin{tabular}{c|c|c}
%\hline
%Parabolic subalgebra $\fq$
%&$\Omega_2|_{V(\mu+\geg)^*}$ & $\Omega_2|_{V(\mu+\geng)^*}$\\
%\hline
%$B_n(i), 3\leq i \leq n-2$ & $M_\fq \big[\C_{-(n- i - \frac{1}{2})\gl_i}\big]$ 
%& $M_\fq[\C_{-\gl_i}]$  \\
%$B_n(n-1)$ & $M_\fq \big[\C_{-\frac{1}{2}\gl_{n-1}} \big]$ & $?$  \\
%$B_n(n)$ & $M_\fq[\C_{\gl_n}]$ & $-$  \\
%$C_n(i), 2 \leq i \leq n-1$  &$?$ & $M_\fq[\C_{\gl_i}]$ \\
%$D_n(i), 3 \leq i \leq n-3$ & $M_\fq[\C_{-(n - i - 1)\gl_i}]$
%& $M_\fq[\C_{-\gl_i}]$\\
%$E_6(3)$ & $M_\fq[\C_{-\gl_3}]$ & $M_\fq[\C_{-2\gl_3}]$  \\
%$E_6(5)$ & $M_\fq[\C_{-\gl_3}]$ & $M_\fq[\C_{-2\gl_3}]$  \\
%$E_7(2)$ & $M_\fq[\C_{-2\gl_2}]$ & $-$ \\
%$E_7(6)$ & $M_\fq[\C_{-\gl_6}]$ & $M_\fq[\C_{-3\gl_6}]$  \\
%$E_8(1)$ & $M_\fq[\C_{-3\gl_1}]$ & $-$ \\
%$F_4(4)$ & $M_\fq[\C_{\gl_4}]$ & $-$ \\
%\hline
%\end{tabular}\label{T93-2}
%\end{center}
%\end{table}

%%%%%%%%%%%%%%%%%%%%%%%%%%%%%%%%%%%%%%%%%%

\vsp

The goal of this paper is  to determine whether or not the maps 
$\varphi_{\Omega_k}$ are standard 
in the quasi-Heisenberg setting. %for certain $\Omega_k$ systems.
To do so, it is convenient to parametrize generalized Verma modules 
by their infinitesimal characters. Therefore, for the rest of this paper,
we write 
\begin{equation}\label{Eqn:InfChar1}
M_\fq[F(\Omega_k|_{W^*})] = M_\fq(\nu-s_0\gl_\fq + \rho)
\end{equation}
and
\begin{equation}\label{Eqn:InfChar2}
M_\fq[\C_{-s_0\gl_\fq}] = M_\fq(-s_0\gl_\fq+\rho),
\end{equation}
where $\rho$ is half the sum of the positive roots.
Then (\ref{Eqn823}) is expressed by
\begin{align} \label{Eqn:GVM2}
M_\fq(\nu-s_0\gl_\fq + \rho) &\stackrel{\varphi_{\Omega_k}}{\to} M_\fq(-s_0\gl_\fq + \rho)\\
u\otimes v &\mapsto u\cdot \iota(v) \nonumber
\end{align}
with $v =  \omega_k(Y^*) \otimes 1$.

%%%%%%%%%%%%%%%%%%%%%%%%%%%%%%%%%%%%%%%%%
\section{Standard maps between generalized Verma modules} 
\label{SS:Hom2}

%To determine whether $\varphi_{\Omega_k}$ is standard or not, 
%it is important to observe the standard maps between generalized Verma modules. 
The aim of this sections is to discuss standard maps 
between generalized Verma modules and 
homomorphisms between (ordinary) Verma modules.
In particular, we specialize a result of Lepowsky to the present situation.

We start with recalling the notion of standard maps. 
For $\eta \in \fh^*$, let $M(\eta)$ be the (ordinary) Verma module 
with highest weight $\eta-\rho$.
Write
\begin{equation*}
\mathbf{P}^+_{\fl} = \{ \zeta \in \fh^* \;|\; \IP{\zeta}{\ga^{\vee}} \in 1 + \Z_{\geq 0} 
\text{ for all } \ga \in \Pi(\fl) \}.
\end{equation*}

For $\eta, \zeta \in \mathbf{P}^+_{\fl}$, 
suppose that there exists a non-zero $\Cal{U}(\fg)$-homomorphism 
$\varphi: M(\eta) \to M(\zeta)$.
If $K(\eta)$ is the kernel of the canonical 
projection map $\text{pr}_\eta: M(\eta) \to M_\fq(\eta)$ then, 
by Proposition 3.1 in \cite{Lepowsky77},  
we have $\gp(K(\eta)) \subset K(\zeta)$. 
Thus the map $\varphi$ induces
a $\Cal{U}(\fg)$-homomorphism $\varphi_{std}:M_\fq(\eta) \to M_\fq(\zeta)$
so that the diagram
\begin{equation*}\label{Eqn:Diagram}
\xymatrix{ 
M(\eta) \ar[r]^\varphi \ar[d]_{\text{pr}_\eta} &
M(\zeta) \ar[d]^{\text{pr}_\zeta}\\
M_\fq(\eta) \ar[r]^{\varphi_{std}}&
M_\fq(\zeta)
}
\end{equation*}
commutes.
The map $\varphi_{std}$ is called the \textbf{standard map} 
from $M_\fq(\eta)$ to $M_\fq(\zeta)$. These maps were
first studied by Lepowsky (\cite{Lepowsky77}).
As $\dim\Hom_{\Cal{U}(\fg)}(M(\eta), M(\zeta)) \leq 1$, 
the standard maps $\varphi_{std}$ are uniquely determined up to 
scalar multiples. 
Note that the standard maps $\varphi_{std}$ could be zero
and also that not every homomorphism between generalized Verma modules 
is standard. Any homomorphisms that are not standard are called non-standard
maps.

\vsp
If $\nu = -(1-s_0)\ga_\fq$ in (\ref{Eqn:GVM2}) with $1-s_0  \in 1 + \Z_{\geq 0}$ then
one can show that the standard map $\varphi_{std}$
from $M_\fq(-(1-s_0)\ga_\fq-s_0\gl_\fq + \rho)$ to $M_\fq(-s_0\gl_\fq+\rho)$ 
is non-zero by computing $\varphi_{std}(1\otimes v^+)$, where $1\otimes v^+$
is a highest weight vector of $M_\fq(-(1-s_0)\ga_\fq-s_0\gl_\fq + \rho)$
with weight $-(1-s_0)\ga_\fq-s_0\gl_\fq$.
To prove it, we will use the following well-known result.
(See for example \cite[Proposition 1.4]{Hum08}.)

\begin{Prop}\label{Prop:Map}
Given $\gl \in \fh^*$ and $\ga \in \Pi$, 
suppose that $n = \IP{\gl+\rho}{\cga} \in 1 + \Z_{\geq 0}$.
If $1\otimes v^+$ is a highest weight vector of weight $\gl$ in $M(\gl+\rho)$ 
then $X_{-\ga}^n \cdot (1\otimes v^+)$ is a highest weight vector of weight
$-n\ga + \gl$.
\end{Prop}

Observe that, by (\ref{Eqn:InfChar1}) and (\ref{Eqn:InfChar2}),
we have $M_\fq(\nu-s_0\gl_\fq+\rho) = 
\Cal{U}(\fg) \otimes_{\Cal{U}(\fq)}F(\Omega_k|_{W^*})$ and 
$M_\fq(-s_0\gl_\fq+\rho) = \Cal{U}(\fg) \otimes_{\Cal{U}(\fq)} \C_{-s_0\gl_\fq}$.
Thus if $v_h$ and $1_{-s_0\gl_\fq}$ are highest weight vectors 
for $F(\Omega_k|_{W^*})$ and $\C_{-s_0\gl_\fq}$, respectively, 
then $1 \otimes v_h$ and $1 \otimes 1_{-s_0\gl_\fq}$
are highest weight vectors for $M_\fq(\nu-s_0\gl_\fq+\rho)$
with highest weight $\nu-s_0\gl_\fq$ and for $M_\fq(-s_0\gl_\fq+\rho)$
with highest weight $-s_0\gl_\fq$, respectively.

\begin{Prop} \label{Prop:Std2}
If $1-s_0 \in 1 + \Z_{\geq 0}$ then
the standard map $\varphi_{std}$ from 
$M_\fq(-(1-s_0)\ga_\fq-s_0\gl_\fq + \rho)$ to $M_\fq(-s_0\gl_\fq + \rho)$ 
maps 
\begin{equation*}
1 \otimes v_h \mapsto
cX_{-\ga_\fq}^{1-s_0} \otimes 1_{-s_0\gl_\fq} \neq 0
\end{equation*}
for some non-zero constant $c$. 
In particular, the standard map $\varphi_{std}$ is non-zero.
\end{Prop}

\begin{proof}
Write $n=1-s_0$ and denote by 
$1\otimes 1_{-n\ga_\fq-s_0\gl_\fq}$ a highest weight vector
for $M(-n\ga_\fq-s_0\gl_\fq + \rho)$ with highest weight $-n\ga_\fq-s_0\gl_\fq$.
Observe that since
$\IP{\gl_\fq}{\cga_0}=\IP{\rho}{\cga_0} = 1$, we have
$n= 1-s_0 = \IP{-s_0\gl_\fq+\rho}{\cga_0}$.
Hence 
$-n\ga_\fq-s_0\gl_\fq+\rho = s_{\ga_\fq}(-s_0\gl_\fq + \rho)$.
By hypothesis, we have $n=1-s_0 \in 1 + \Z_{\geq 0}$.
It then follows from Proposition \ref{Prop:Map} that
the map $\varphi: 
M_\fq(-n\ga_\fq-s_0\gl_\fq + \rho) \to M_\fq(-s_0\gl_\fq + \rho)$ 
is given by
\begin{equation*}
\varphi(1\otimes 1_{-n\ga_\fq-s_0\gl_\fq}) 
= c X_{-\ga_\fq}^n \otimes 1
\end{equation*}
with $c \neq 0$.
As $\ga_\fq \in \Pi \backslash \Pi(\fl)$, 
if $\text{pr}_{-s_0\gl_\fq+\rho}: 
M(-s_0\gl_\fq + \rho) \to M_\fq(-s_0\gl_\fq + \rho)$ is the 
canonical projection map
then $\text{pr}_{-s_0\gl_\fq+\rho}(X_{-\ga_\fq}^n \otimes 1) \neq 0$.
Then the universal property of $M_\fq(-n\ga_\fq-s_0\gl_\fq+\rho)$ 
in the relative category $\Cal{O}^\fq$ 
(see for example Section 9.4 in \cite{Hum08})
guarantees that 
$\text{pr}_{-s_0\gl_\fq+\rho} \circ \varphi$ factors through a non-zero map
$\varphi_{std}: M_\fq(-n\ga_\fq-s_0\gl_\fq+\rho) \to M_\fq(-s_0\gl_\fq + \rho)$.
\end{proof}

\vsp

In order to determine if $\varphi_{std}$ is non-zero in a more 
general setting, we will use the following theorem by Lepowsky.
As usual, if there is a non-zero $\Cal{U}(\fg)$-homomorphism
from $M(\eta)$ into $M(\zeta)$ then we write $M(\eta) \subset M(\zeta)$.

\begin{Thm}\cite[Proposition 3.3]{Lepowsky77}\label{Thm:Std}
Let $\eta, \zeta \in \mathbf{P}^+_{\fl}$, and assume that $M(\eta) \subset M(\zeta)$. 
Then the standard map $\varphi_{std}$ from $M_\fq(\eta)$ to $M_\fq(\zeta)$ is zero
if and only if $M(\eta) \subset M(s_\ga \zeta)$ for some $\ga \in \Pi(\fl)$.
\end{Thm}

Theorem \ref{Thm:Std} reduces the existence problem of the non-zero standard
map $\varphi_{std}$ between generalized Verma modules to that of 
the non-zero map between appropriate Verma modules. It is well known 
when a non-zero $\Cal{U}(\fg)$-homomorphism between Verma modules exists.
To describe the condition efficiently,
we first introduce the definition of a \emph{link} of two weights.

\begin{Def}\label{Def:Link}(Bernstein-Gelfand-Gelfand)
Let $\gl, \gd \in \fh^*$ and $\gb_1, \ldots, \gb_t \in \gD^+$. Set $\gd_0 = \gd$
and $\gd_i = s_{\gb_i} \cdots s_{\gb_1}\gd$ for $1 \leq i \leq t$.
We say that the sequence $(\gb_1, \ldots, \gb_t)$ \textbf{links} $\gd$ to $\gl$ if 
\begin{enumerate}
\item[(1)] $\gd_t = \gl$ and
\item[(2)] $\IP{\gd_{i-1}}{\gb_i^{\vee}} \in \Z_{\geq 0}$ for $1\leq i \leq t$.
\end{enumerate}
\end{Def}

\begin{Thm}\label{Thm:Link}(BGG-Verma)
Let $\gl, \gd \in \fh^*$. The following conditions are equivalent:
\begin{enumerate}
\item[(1)] $M(\gl) \subset M(\gd)$
\item[(2)] $L(\gl)$ is a composition factor of $M(\gd)$
\item[(3)] There exists a sequence $(\gb_1, \ldots, \gb_t)$ with $\gb_i \in \gD^+$
that links $\gd$ to $\gl$,
\end{enumerate}
where $L(\gl)$ is the unique irreducible quotient of $M(\gl)$.
\end{Thm}

Observe that if 
there is a non-zero $\Cal{U}(\fg)$-homomorphism 
(not necessarily standard)
from $M_\fq(\eta)$ to $M_\fq(\zeta)$ 
then $M(\eta) \subset M(\zeta)$.
%Indeed, if there exists a non-zero $\Cal{U}(\fg)$-homomorphism
%$f: M_\fq(\eta) \to M_\fq(\zeta)$ then the quotient
%$M_\fq(\eta)/\ker(f)$ is embedded into $M_\fq(\zeta)$.
%Observe that, as $L(\eta)$ is a unique irreducible quotient of $M(\eta)$,
%it is also a unique irreducible quotient of $M_\fq(\eta)$ and so of 
%$M_\fq(\eta)/\ker(f)$. In particular, via the embedding 
%$M_\fq(\eta)/\ker(f) \hookrightarrow M_\fq(\zeta)$, the irreducible 
%quotient $L(\eta)$ is a composition factor of $M_\fq(\zeta)$. 
%Since the composition factors of $M_\fq(\zeta)$ are those of $M(\zeta)$, 
%this shows that $L(\zeta)$ is a composition factor of $M(\zeta)$. 
%Now it follows from Theorem \ref{Thm:Link} that $M(\eta) \subset M(\zeta)$.
By taking into account Theorem \ref{Thm:Link} and this observation,
in our setting, Theorem \ref{Thm:Std} is equivalent to the following 
proposition.

\begin{Prop}\label{Prop:Std}
Let $M_\fq(\nu-s_0\gl_\fq + \rho)$ and $M_\fq(-s_0\gl_\fq + \rho)$ be
the generalized Verma modules in (\ref{Eqn:GVM2}).
Then the standard map from $M_\fq(\nu-s_0\gl_\fq + \rho)$ 
to $M_\fq(-s_0\gl_\fq + \rho)$ is zero if and only if 
there exists $\ga \in \Pi(\fl)$ so that $-\ga - s_0\gl_\fq +\rho$
is linked to $\nu-s_0\gl_\fq +\rho$.
\end{Prop}

\begin{proof}
First observe that 
since there exists a non-zero $\Cal{U}(\fg)$-homomorphism
$\varphi_{\Omega_k}$ from $M_\fq(\nu-s_0\gl_\fq + \rho)$ 
to $M_\fq(-s_0\gl_\fq + \rho)$,  
we have $M(\nu-s_0\gl_\fq + \rho) \subset M(-s_0\gl_\fq + \rho)$.
Therefore, by Theorem \ref{Thm:Std} and Theorem \ref{Thm:Link},
the standard map from $M_\fq(\nu-s_0\gl_\fq + \rho)$ 
to $M_\fq(-s_0\gl_\fq + \rho)$ is zero if and only if
there exists $\ga \in \Pi(\fl)$ so that $s_\ga(-s_0\gl_\fq +\rho)$
is linked to $\nu- s_0\gl_\fq +\rho$.
As $\IP{\gl_\fq}{\ga^{\vee}} = 0$ and $\IP{\rho}{\ga^\vee} = 1$
for $\ga \in \Pi(\fl)$, we have
$s_\ga(-s_0\gl_\fq +\rho) = -\ga -s_0\gl_\fq+\rho$.
Now this proposition follows.
\end{proof}

With Proposition \ref{Prop:Std} in hand,
in the next two sections,
we shall determine whether or not the homomorphisms 
$\varphi_{\Omega_k}$ that arise from 
the $\Omega_k$ system(s) for $k=1,2$ constructed
in \cite{KuboThesis1} are standard.
%when $\fq$ is a two-step nilpotent maximal parabolic 
%subalgebra of non-Heisenberg type.
%Therefore, for the rest of this paper, parabolic subalgebra $\fq$ is 
%assumed to be a two-step nilpotent maximal parabolic subalgebra 
%of non-Heisenberg type, unless otherwise specified.

%%%%%%%%%%%%%%%%%%%%%%%%%%%%%%%%%%%%%%%%% 
\section{The homomorphism $\varphi_{\Omega_1}$ 
induced by the $\Omega_1$ system}\label{SS:Omega1Hom}

In this section we show that the homomorphism $\varphi_{\Omega_1}$
arising from the $\Omega_1$ system associated to a 
maximal parabolic subalgebra $\fq$ of quasi-Heisenberg type 
is standard.
For each $\ga \in \gD^+$,  we define
$\{X_\ga, X_{-\ga}, H_{\ga} \}$ as an $\f{sl}(2)$-triple;
in particular, we have $[X_\ga, X_{-\ga}] = H_\ga$.
For $\ga, \gb \in \gD$ with $\ga + \gb \in \gD$, 
we write a constant $N_{\ga,\gb}$ for
$[X_\ga, X_\gb] = N_{\ga,\gb} X_{\ga + \gb}$.
Recall from Subsection 2.2
that an irreducible constituent $W$ of $\fg(-r + k) \otimes \fg(r)$ is called
special for $\tau_k$ if $\ttau_k|_{W^*} \not \equiv 0$.

\vsp
It follows from (\ref{Eqn:TauTwoStep}) that 
the $\Omega_1$ system is constructed from the map
$\tau_1 : \fg(1) \to \fg(-1) \otimes \fz(\fn)$
with $X \mapsto \big( \ad(X)\otimes \text{Id}\big)\omega$,
%the $L$-equivariant polynomial map 
%\begin{align*}
%\tau_1 : \fg(1) &\to \fg(-1) \otimes \fz(\fn)\\
%X &\mapsto \ad(X)\omega
%\end{align*}
where $\omega = \sum_{\gamma_j \in \gD(\fz(\fn))}X_{-\gamma_j} \otimes X_{\gamma_j}$.
In Section 5 of \cite{KuboThesis1}, it is shown that 
irreducible constituent $W$ of $\fg(-1) \otimes \fz(\fn)$ is special
if and only if $W \cong \fg(1)$ and also that there is only unique 
such a constituent. 
Via the composition of maps in (\ref{Eqn:Comp}),
the $\Omega_1$ system is given by 
$R(X_{-\ga_1}), \ldots, R(X_{-\ga_m})$ 
for $\gD(\fg(1)) = \{\ga_1, \ldots, \ga_m\}$.

\begin{Thm}\cite[Theorem 5.7]{KuboThesis1}\label{Thm8.1.2}
Let $\fg$ be a complex simple Lie algebra, 
and let $\fq$ be a maximal parabolic subalgebra 
of quasi-Heisenberg type. 
Then the $\Omega_1$ system
is conformally invariant on $\Cal{L}_{s}$ if and only if $s=0$.
\end{Thm}

It follows from Proposition \ref{Cor243} and 
Theorem \ref{Thm8.1.2} that the $\Omega_1$ system 
yields a finite dimensional simple $\fl$-submodule $F(\Omega_1)$ in 
$\big( \Cal{U}(\fg)\otimes_{\Cal{U}(\fq)}\C_0 \big)^{\fn} = M_\fq(\rho)^\fn$. 
If $\ga_{\fq}$ is the simple root that 
determines the maximal parabolic subalgebra $\fq$
then, as it is the lowest weight for $\fg(1)$,
$W^*\cong \fg(-1)$ has highest weight $-\ga_\fq$.
Thus, by Proposition \ref{Cor243},
the simple $\fl$-module $F(\Omega_1)$ has highest weight 
$\nu-s_0\gl_\fq=-\ga_\fq$.
Now, by (\ref{Eqn:GVM2}),
the inclusion map $F(\Omega_1) \hookrightarrow M_\fq(\rho)$ 
induces a non-zero $\Cal{U}(\fg)$-homomorphism
\begin{equation*}
\varphi_{\Omega_1}: M_\fq(-\ga_\fq + \rho) \to M_\fq(\rho).
\end{equation*}

\begin{Prop}\label{Prop:StdO1}
If $\fq$ is a maximal parabolic subalgebra 
of quasi-Heisenberg type then 
the standard map $\varphi_{std}: M_\fq(-\ga_\fq + \rho)
\to M_\fq(\rho)$ is non-zero. 
%Moreover, there exists $c \neq 0$ so that 
%$\varphi_{std}(1\otimes v_h) = c X_{-\ga_\fq}\otimes 1_0$.
\end{Prop}

\begin{proof}
This follows from Proposition \ref{Prop:Std2}
with $s_0 = 0$.
\end{proof}

\begin{Thm}\label{Thm:MapO1}
If $\fq$ is a maximal parabolic subalgebra 
of quasi-Heisenberg type then 
the map $\varphi_{\Omega_1}$ is standard.
\end{Thm}

\begin{proof}
Let $v_h$ be a highest weight vector for $F(\Omega_1)$.
Since $\varphi_{\Omega_1} ( 1 \otimes v_h) = 1 \cdot v_h = v_h$,
to prove that $\varphi_{\Omega_1}$ is standard, 
by Propostion \ref{Prop:Std2} and Proposition \ref{Prop:StdO1},
it suffices to show that
$v_h = cX_{-\ga_\fq} \otimes 1_0$ with some non-zero constant $c$.
To do so, as $v_h$ is a highest weight vector for $F(\Omega_1)$,
we show that $X_{-\ga_\fq}\otimes 1_0$ is a highest weight vector
for $F(\Omega_1)$.
Since the $\Omega_1$ system is 
$R(X_{-\ga_1}), \ldots, R(X_{-\ga_m})$ for 
$\gD(\fg(1)) = \{\ga_1, \ldots, \ga_m\}$,
it is clear that the elements $\omega_1(X_{-\ga_j}) \in \gs(\Sym^1(\bar \fn))=\bar \fn$
that correspond to $R(X_{-\ga_j})$ under $R$
are $\omega_1(X_{-\ga_j}) = X_{-\ga_j}$. Then it follows
from (\ref{Eqn:Ltype1}) that
\begin{equation*}
F(\Omega_1) =\text{span}_\C\{
X_{-\ga} \otimes 1_0 \; | \; \ga \in \gD(\fg(1))\}.
\end{equation*}
Therefore $X_{-\ga_\fq}\otimes 1_0$ is a highest weight vector for $F(\Omega_1)$.
\end{proof}

%%%%%%%%%%%%%%%%%%%%%%%%%%%%%%%%%%%%%%%%%
\section{The homomorphisms $\varphi_{\Omega_2}$ 
induced by the $\Omega_2$ systems} 
\label{SS:Omega2Hom} 
 
The aim of this section is to classify  
the homomorphisms $\varphi_{\Omega_2}$ 
that are induced by the $\Omega_2$ systems associated to 
maximal parabolic subalgebras $\fq$ listed 
in (\ref{Eqn4.0.1}) and (\ref{Eqn4.0.2}) 
as standard or not. 

We first recall from Section 6 of \cite{KuboThesis1}
some observation on special constituents.
The $\Omega_2$ systems are constructed from the map
$\tau_2 : \fg(1) \to \fl \otimes \fz(\fn)$ with
$X \mapsto \frac{1}{2}\big( \ad(X)^2\otimes \text{Id}\big)\omega$.
%the $L$-equivariant polynomial map
%\begin{align*}
%\tau_2 : \fg(1) &\to \fl \otimes \fz(\fn)\\
%X &\mapsto \frac{1}{2}\ad(X)^2\omega.
%\end{align*}
Observe that if $V(\nu)$ is a special constituent 
of $\fg(0) \otimes \fz(\fn) = \fl \otimes \fz(\fn)$ with 
highest weight $\nu$ then,
as $V(\nu)^*$ is embedded into 
$\Cal{P}^2(\fg(1)) \cong \Sym^2(\fg(1))^* \subset \fg(1)^*\otimes \fg(1)^*$,
we have $V(\nu) \hookrightarrow \fg(1)\otimes \fg(1)$.
Thus the highest weight $\nu$ is of the form $\mu +\ge$, where 
$\mu$ is the highest weight for $\fg(1)$ and $\ge$ is some weight 
for $\fg(1)$.

Recall from (\ref{Eqn4.1.8}) that we have 
$\fl = \C H_\fq \oplus \flg \oplus \flng$.
Thus the tensor product $\fl \otimes \fz(\fn)$ may be written as
%\begin{equation*}
$\fl \otimes \fz(\fn) 
= \big( \C H_\fq\otimes \fz(\fn) \big) 
\oplus \big( \flg \otimes \fz(\fn) \big) 
\oplus \big( \flng\otimes \fz(\fn) \big)$. 
%\end{equation*}
It is shown in Section 6 of \cite{KuboThesis1} that,
for $\fq$ under consideration in (\ref{Eqn4.0.1}) and (\ref{Eqn4.0.2}), 
there are exactly one or two special constituents
of $\fl \otimes \fz(\fn)$; one is an irreducible constituent of
$\flg\otimes \fz(\fn)$ and the other is equal to $\flng \otimes \fz(\fn)$.
We denote by $V(\mu+\geg)$ and $V(\mu+\geng)$
the special constituents so that 
$V(\mu+\geg) \subset \flg\otimes \fz(\fn)$ and
$V(\mu+\geng) = \flng\otimes \fz(\fn)$.
We summarize the data on the special constituents 
in Table \ref{T51} and Table \ref{T52} below.
We use the standard realizations for the roots for the classical algebras,
while the Bourbaki conventions \cite{Bourbaki08} are used for the exceptional algebras
for the labels of the simple roots. 
A dash in the column for $V(\mu + \geng)$
indicates that $\flng = \{0\}$ for the case.
(So there is no special constituent $V(\mu+\geng)$.)

\begin{table}[h]
\caption{Highest Weights for Special Constituents (Classical Cases)}
\begin{center}
\begin{tabular}{ccccc} 
\hline
Type        &$V(\mu + \geg)$ &$V(\mu + \geng)$ \\
\hline
$B_n(i),\;  3\leq i \leq n-2$ &$2\varepsilon_1$&$\varepsilon_1 + \varepsilon_2 + \varepsilon_{i+1} + \varepsilon_{i+2}$\\
$B_n(n-1)$ &$2\varepsilon_1$&$\varepsilon_1 + \varepsilon_2 + \varepsilon_n$\\
$B_n(n)$ &$2\varepsilon_1$&$-$\\
$C_n(i), \; 2\leq i \leq n-1$ &$\varepsilon_1 + \varepsilon_2$&$2\varepsilon_1 + 2\varepsilon_{i+1}$\\
$D_n(i), \; 3\leq i \leq n-3$ &$2\varepsilon_1$&$\varepsilon_1 +\varepsilon_2 + \varepsilon_{i+1} + \varepsilon_{i+2}$\\
\hline
\end{tabular} \label{T51}
\end{center}
\end{table}

\begin{table}[h]
\caption{Highest Weights for Special Constituents (Exceptional Cases)}
\begin{center}
\makebox[0.93 \width][r]{ 
\begin{tabular}{cc} 
\hline
Type        &$V(\mu + \geg)$ \\
\hline
$E_6(3)$ &$\ga_1 + 2\ga_2 + 2\ga_3 + 4\ga_4 + 3\ga_5 + 2\ga_6$\\
$E_6(5)$ &$2\ga_1 + 2\ga_2 + 3\ga_3 + 4\ga_4 + 2\ga_5 + \ga_6$\\
$E_7(2)$ &$2\ga_1 + 2\ga_2 + 4\ga_3 + 5\ga_4 + 4\ga_5 + 3\ga_6 + 2\ga_7$\\
$E_7(6)$ &$2\ga_1 + 3\ga_2 + 4\ga_3 + 6\ga_4 + 4\ga_5 + 2\ga_6 + \ga_7$ \\
$E_8(1)$ &$2\ga_1 + 4\ga_2  + 5\ga_3 + 8\ga_4 + 7\ga_5 + 6\ga_6 + 4\ga_7 + 2\ga_8$\\
$F_4(4)$ & $2\ga_1+ 4\ga_2 + 6\ga_3 + 2\ga_4$\\
\hline
%              & \\
%              & \\
%\hline              
Type       &$V(\mu + \geng)$ \\
\hline
$E_6(3)$ &$2\ga_1 + 2\ga_2 + 2\ga_3 + 3\ga_4 + 2\ga_5 + \ga_6$\\
$E_6(5)$ &$\ga_1 + 2\ga_2 + 2\ga_3 + 3\ga_4 + 2\ga_5 + 2\ga_6$ \\
$E_7(2)$ &$-$\\
$E_7(6)$ &$2\ga_1 + 2\ga_2 + 3\ga_3 + 4\ga_4 + 3\ga_5 + 2\ga_6 + 2\ga_7$ \\
$E_8(1)$ &$-$\\
$F_4(4)$ &$-$\\
\hline
\end{tabular} \label{T52}
}
\end{center}
\end{table}

\begin{Def}\cite[Definition 6.20]{KuboThesis1}\label{Def5.2.8}
Let $\mu$ be the highest weight for $\fg(1)$, and let 
$\ge = \geg$ or $\geng$.
We say that a special constituent $V(\mu+\ge)$ is of 
\begin{enumerate}
\item [(1)] \textbf{type 1a} if $\mu + \ge$ is not a root with $\ge \neq \mu$ and both 
$\mu$ and $\ge$ are long roots,
\item [(2)] \textbf{type 1b} if $\mu + \ge$ is not a root with $\ge \neq \mu$ and either
$\mu$ or $\ge$ is a short root,
\item [(3)] \textbf{type 2} if $\mu + \ge = 2\mu$ is not a root, or
\item [(4)] \textbf{type 3} if $\mu + \ge$ is a root.
\end{enumerate}
\end{Def}

Table \ref{T55} below shows the types of special constituents 
for each maximal parabolic subalgebra $\fq$.
In \cite{KuboThesis1}  the special values 
for the type 1a and type 2 cases are determined.
\begin{table}[h]
\caption{Types of Special Constituents}
\begin{center}
\begin{tabular}{ccccc} 
\hline
Type        &$V(\mu+\geg)$ &$V(\mu+\geng)$ \\
\hline
$B_n(i),\;  3\leq i \leq n-2$ &Type 1a&Type 1a\\
$B_n(n-1)$                      &Type 1a&Type 1b\\
$B_n(n)$                         &Type 2&$-$\\
$C_n(i), \; 2\leq i \leq n-1$ &Type 3&Type 2\\
$D_n(i), \; 3\leq i \leq n-3$ &Type 1a&Type 1a\\
$E_6(3)$                          &Type 1a&Type 1a\\
$E_6(5)$                          &Type 1a&Type 1a\\
$E_7(2)$                          &Type 1a&$-$\\
$E_7(6)$                          &Type 1a&Type 1a\\
$E_8(1)$                          &Type 1a&$-$\\
$F_4(4)$                          &Type 2&$-$\\
\hline
\end{tabular} \label{T55}
\end{center}
\end{table}

For $\mu+\ge = \mu+\geg$ or $\mu+\geng$,
we write
\begin{equation*} 
\gD_{\mu+\ge}(\fg(1)) = \{\ga \in \gD(\fg(1))\; | \; \mu+\ge - \ga \in \gD(\fg(1))\}.
\end{equation*}
We denote by $|\gD_{\mu+\ge}(\fg(1))|$ the number of elements 
in $\gD_{\mu+\ge}(\fg(1))$.

\begin{Thm}\cite[Theorem 7.16, Corollary 7.23]{KuboThesis1}\label{Thm8.3.1}
Suppose that $V(\mu+\ge)$ is a special constituent of type 1a or type 2.
\begin{enumerate}
\item[(1)] If $V(\mu+\ge)$ is of type 1a then
the $\Omega_2|_{V(\mu+\ge)^*}$ system 
is conformally invariant on $\Cal{L}_{s}$
if and only if 
\begin{equation*} \label{EqnSpecial}
s = \frac{|\gD_{\mu+\ge}(\fg(1))|}{2}-1.
\end{equation*}

\item[(2)] If $V(\mu+\ge)$ is of type 2 then
the $\Omega_2|_{V(\mu+\ge)^*}$ system
is conformally invariant on $\Cal{L}_{s}$
if and only if 
\begin{equation*}
s=-1.
\end{equation*}
\end{enumerate}
\end{Thm}

Let $\gl_i$ be the fundamental weight for the simple root 
$\ga_i$ that determines the maximal parabolic subalgebra $\fq$.
Table \ref{T93} below summarizes
the line bundles $\Cal{L}_s = \Cal{L}(s\gl_i)$
on which the $\Omega_2$ systems are conformally invariant.
\begin{table}[h]
\caption{Line bundles with special values}
\begin{center}
\begin{tabular}{c|c|c}
\hline
Parabolic $\fq$ 
&$\Omega_2|_{V(\mu+\geg)^*}$ & $\Omega_2|_{V(\mu+\geng)^*}$\\
\hline
$B_n(i), 3\leq i \leq n-2$ & $\Cal{L}\big( (n- i - \frac{1}{2})\gl_i\big)$ & $\Cal{L}(\gl_i)$  \\
$B_n(n-1)$ & $\Cal{L}\big( \frac{1}{2}\gl_{n-1} \big)$ & $?$  \\
$B_n(n)$ & $\Cal{L}(-\gl_n)$ & $-$  \\
$C_n(i), 2 \leq i \leq n-1$  &$?$ & $\Cal{L}(-\gl_i)$ \\
$D_n(i), 3 \leq i \leq n-3$ & $\Cal{L}\big((n - i - 1)\gl_i\big)$& $\Cal{L}(\gl_i)$\\
$E_6(3)$ & $\Cal{L}(\gl_3)$ & $\Cal{L}(2\gl_3)$  \\
$E_6(5)$ & $\Cal{L}(\gl_5)$ & $\Cal{L}(2\gl_5)$  \\
$E_7(2)$ & $\Cal{L}(2\gl_2)$ & $-$ \\
$E_7(6)$ & $\Cal{L}(\gl_6)$ & $\Cal{L}(3\gl_6)$  \\
$E_8(1)$ & $\Cal{L}(3\gl_1)$ & $-$ \\
$F_4(4)$ & $\Cal{L}(-\gl_4)$ & $-$ \\
\hline
\end{tabular}\label{T93}
\end{center}
\end{table}
When $\fq$ is of type $B_n(n-1)$,
the constituent $V(\mu+\geng)$ is of type 1b,
and when $\fq$ is of type $C_n(i)$,
the constituent $V(\mu+\geg)$ is of type 3. 
Therefore, a question mark is put for these cases in the table.

\vsp

Now, with the results in Table \ref{T93} in hand, 
we determine the standardness of 
$\varphi_{\Omega_2}$.
Observe from Table \ref{T55} and Table \ref{T93} that 
each $\Omega_2|_{V(\mu+\ge)^*}$ system satisfies  
exactly one of the following:
\begin{enumerate}
\item The special constituent $V(\mu+\ge)$ is of type 2.
\item The special value $s_0$ is a positive integer.
\item The parabolic subalgebra $\fq$ is of type $B_n(i)$ 
for $3\leq i \leq n-1$ and $V(\mu+\ge) = V(\mu+\geg)$.
\end{enumerate}
We shall consider these three cases separately.

%%%%%%%%%%%%%%%%%%%%%%%%%%%%%%%%%%%%%%%%%
\subsection{The type 2 case} 

We first study the homomorphism attached to 
the special constituent $V(\mu+\ge)$ of type 2.
By Table \ref{T55}, we consider the following three cases: 
\begin{equation*}
V(\mu+\geg) \text{ for } B_n(n), \;
V(\mu+\geng) \text{ for } C_n(i) (2 \leq i \leq n-1),\; \text{and} \;
V(\mu+\geg) \text{ for } F_4(4).
\end{equation*}

If $V(\mu+\ge)$ is a type 2 special constituent then,
by definition, $V(\mu+\ge) = V(2\mu)$.
Thus, as $\mu$ and $\ga_\fq$ are the highest and lowest 
weights for $\fg(1)$, respectively, we have
$V(\mu+\ge)^* = V(2\mu)^* =V(-2\ga_\fq)$.
Therefore $\nu$ in (\ref{Eqn:GVM2}) is $\nu=-2\ga_\fq$.
Moreover, by Theorem \ref{Thm8.3.1}, 
the $\Omega_2|_{V(2\mu)^*}$ system is conformally 
invariant on the line bundle $\Cal{L}(-\gl_\fq)$. 
Thus $s_0 = -1$. Therefore it follows from (\ref{Eqn:GVM2})
that we have
\begin{equation*}
\varphi_{\Omega_2}: M_\fq(-2\ga_\fq +\gl_\fq+ \rho) \to M_\fq(\gl_\fq+\rho).
\end{equation*}

\begin{Prop}\label{Prop:StdType2}
If $\fq$ is the maximal parabolic subalgebra of type 
$B_n(n)$, $C_n(i)$ for $2\leq i \leq n-1$, or $F_4(4)$
then the standard map $\varphi_{std}$ from $M_\fq(-2\ga_\fq+\gl_\fq + \rho)$ 
to $M_\fq(\gl_\fq+\rho)$ is non-zero. 
%Moreover, 
%there exists $c \neq 0$ so that 
%$\varphi_{std}(1\otimes v_h) =
%cX^2_{-\ga_\fq} \otimes 1_{\gl_\fq}$.
\end{Prop}

\begin{proof}
This follows from Proposition \ref{Prop:Std2}
with $s_0 = -1$.
\end{proof}

In Section 7.3 of \cite{KuboThesis1}, 
it is observed that if $Y^*_l$ is a lowest weight vector for $V(2\mu)^*$
then the differential operator $\Omega_2(Y^*_l)$ is of the form
\begin{equation*}\label{Eqn:Type22}
\Omega_2(Y^*_l) = a R(X_{-\mu})^2,
\end{equation*}
for some constant $a$.
Therefore, the element $\omega_2(Y^*_l)$ in 
$\gs(\Sym^2(\bar\fn)) \subset \Cal{U}(\bar{\fn})$
that corresponds to $\Omega_2(Y^*_l)$ under $R$
in (\ref{Eqn:Comp}) is of the form
\begin{equation}\label{Eqn:Lowest}
\omega_2(Y^*_l) = a  X^2_{-\mu}.
\end{equation}
Thus the simple $\fl$-submodule 
$F(\Omega_2|_{V(2\mu)^*})$ of 
$M_\fq(\gl_\fq + \rho)^{\fn}=
\big( \Cal{U}(\fg)\otimes_{\Cal{U}(\fq)}\C_{\gl_\fq}\big)^{\fn}$ 
has lowest weight $X^2_{-\mu} \otimes 1_{\gl_\fq}$.
%with $\mu$ the highest weight for $\fg(1)$.

\begin{Thm}\label{Thm:MapO2Type2}
Let $\fq$ be a maximal parabolic 
subalgebra of quasi-Heisenberg type,
listed in (\ref{Eqn4.0.1}) or (\ref{Eqn4.0.2}).
If the special constituent $V(\mu+\ge)$ is of type 2 then
the map $\varphi_{\Omega_2}$ is standard.
\end{Thm}

\begin{proof}
In order to prove that $\varphi_{\Omega_2}$ is standard, 
by Proposition \ref{Prop:StdType2}, it suffices to show that
$X^2_{-\ga_\fq}\otimes 1_{\gl_\fq}$ is a highest weight vector for 
$F(\Omega_2|_{V(2\mu)^*})$. 
Since $F(\Omega_2|_{V(2\mu)^*})$ has highest weight 
$\nu-s_0\gl_\fq = -2\ga_\fq+\gl_\fq$, it is enough to show that 
$X^2_{-\ga_\fq}\otimes 1_{\gl_\fq}$ is in $F(\Omega_2|_{V(2\mu)^*})$.
We know that
a lowest weight vector for $F(\Omega_2|_{V(2\mu)^*})$
is $X^2_{-\mu} \otimes 1_{\gl_\fq}$. This will allow us to show that 
$X_{-\ga_\fq}^2 \otimes 1_{\gl_\fq}$ is in $F(\Omega_2|_{V(2\mu)^*})$.
We do so in a case-by-case manner. 
%Recall that we have to consider the following three cases:
%\begin{equation*}
%V(\mu+\geg) \text{ for } B_n(n), \;
%V(\mu+\geng) \text{ for } C_n(i) (2 \leq i \leq n-1),\; \text{and} \;
%V(\mu+\geg) \text{ for } F_4(4).
%\end{equation*}
%\begin{enumerate}
%\item $V(\mu+\geg)$ for $B_n(n)$,
%\item $V(\mu+\geng)$ for $C_n(i)$ $(2 \leq i \leq n-1)$, and 
%\item $V(\mu+\geg)$ for $F_4(4)$.
%\end{enumerate} 
Since the arguments are similar for each case, we show only 
the case $V(\mu+\geg)$ for $B_n(n)$. (For the other cases
see Section 8.3 in \cite{KuboThesis}.)
In the standard realization of the roots
we have $\mu=\vep_1$, $\ga_\fq = \ga_n = \vep_n$, and 
\begin{equation*}
\gD^+(\fl) = \{\vep_j-\vep_k \; | \; 1\leq j < k \leq n\}
\end{equation*}
(see Appendix \ref{chap:Data}).
Thus,
\begin{equation*}
X^2_{-\mu}\otimes 1_{\gl_\fq} = X^2_{-\vep_1} \otimes 1_{\gl_n}
\; \text{ and } \;
X^2_{-\ga_\fq}\otimes 1_{\gl_\fq} = X^2_{-\vep_n}\otimes 1_{\gl_n}.
\end{equation*}
A direct computation shows that 
\begin{equation*}
X^2_{\vep_1-\vep_n}\cdot (X^2_{-\vep_1} \otimes 1_{\gl_n})
=2N_{\vep_1 -\vep_n, -\vep_1}^2X^2_{-\vep_n}\otimes 1_{\gl_n},
\end{equation*}
where $N_{\vep_1 -\vep_n, -\vep_1}$ is the constant so that
$[X_{\vep_1 -\vep_n}, X_{-\vep_1}] = N_{\vep_1 -\vep_n, -\vep_1}X_{-\vep_n}$.
(See the beginning of Section \ref{SS:Omega1Hom}.)
Therefore, as $X_{\vep_1-\vep_n} \in \fl$,
we have $X^2_{-\ga_\fq} \otimes 1_{\gl_\fq}
=X^2_{-\vep_n}\otimes 1_{\gl_n} \in F(\Omega_2|_{V(2\mu)^*})$.
\qedhere

\end{proof}

%%%%%%%%%%%%%%%%%%%%%%%%%%%%%%%%%%%%%%%%%
\subsection{The positive integer special value case} 

Next we handle the case that the special value $s_0$
is a positive integer.

\begin{Thm}\label{Thm:MapO2Int}
Let $\fq$ be a maximal parabolic subalgebra of quasi-Heisenberg type,
listed in (\ref{Eqn4.0.1}) or (\ref{Eqn4.0.2}).
If the special value $s_0$ is a positive integer
then the standard map from $M_\fq(\nu-s_0\gl_\fq +\rho)$
to $M_\fq(-s_0\gl_\fq +\rho)$ is zero. Consequently,
the map $\varphi_{\Omega_2}$ is non-standard.
\end{Thm}

\begin{proof}
By Proposition \ref{Prop:Std}, to show that the standard map
is zero, it suffices to show that there exists $\ga \in \Pi(\fl)$
so that $-\ga-s_0\gl_\fq+\rho$ is linked to $\nu-s_0\gl_\fq+\rho$.
We achieve it by a case-by-case observation.
By Table \ref{T93}, the following are the cases under consideration:
\begin{enumerate}
\item $V(\mu+\geng)$ for $B_n(i)$ $(3\leq i \leq n-2)$
\item $V(\mu+ \geg)$ and $V(\mu+\geng)$ for $D_n(i)$
$(3 \leq i \leq n-3)$
\item $V(\mu+\geg)$ and $V(\mu+\geng)$ for $E_6(3)$
\item $V(\mu+\geg)$ and $V(\mu+\geng)$ for $E_6(5)$
\item $V(\mu+\geg)$ for $E_7(2)$
\item $V(\mu+\geg)$ and $V(\mu+\geng)$ for $E_7(6)$
\item $V(\mu+\geg)$ for $E_8(1)$
\end{enumerate}

Our strategy is to first observe that the highest weight $\nu$ 
for $V(\mu+\ge)^*$ is of the form 
\begin{equation*}
\nu=-2\gb-\ga'-\ga''  
\end{equation*}
for some $\gb \in \gD(\fg(1))$ and $\ga', \ga'' \in \Pi(\fl)$.
We then show that
the sequence $(\ga', \gb)$ links $-\ga''-s_0\gl_\fq+\rho$ 
to $(-2\gb-\ga'-\ga'')-s_0\gl_\fq+\rho$. 
Here we only show three cases, namely, 
$V(\mu+\geng)$ for $B_n(i)$,
$V(\mu+ \geg)$ for $D_n(i)$ $(3 \leq i \leq n-3)$,
and $V(\mu+\geg)$ for $E_6(3)$.
Other cases can be shown similarly.
(For some details for the other cases see Section 8.3 in \cite{KuboThesis}.)
\vskip 0.1in

\underline{1. $V(\mu+\geng)$ for $B_n(i)$ for $3\leq i \leq n-2$:}
Since, by Table \ref{T93}, the special value $s_0$ is $s_0 = 1$,
we wish to show that there is $\ga \in \Pi(\fl)$ so that 
$-\ga -\gl_i + \rho$ is linked to $\nu-\gl_i + \rho$.
First we find the highest weight $\nu$ for $V(\mu+\geng)^*$.
Observe that we have $\gD^+(\fl) = \gD^+(\flg) \cup \gD^+(\flng)$
with
\begin{equation*}
\gD^+(\flg) = \{ \vep_j-\vep_k \; | \; 1 \leq j < k \leq i\}
\end{equation*}
and 
\begin{equation*}
\gD^+(\flng) = \{ \vep_j \pm \vep_k \; | \; i+1 \leq j < k \leq n \}
\cup \{ \vep_j \; | \; i+1 \leq j \leq n\}
\end{equation*}
in the standard realization of the roots (see Appendix \ref{chap:Data}).
Since
\begin{equation*}
\gD(\fz(\fn)) = \{ \vep_j + \vep_k \; | \; 1\leq j < k \leq i\},
\end{equation*}
the simple $\fl$-module $\fz(\fn)$ has lowest weight $\vep_{i-1} + \vep_i$.
As $V(\mu+\geng) = \flng \otimes \fz(\fn)$, we have
%\begin{equation*}
$V(\mu+\geng)^*=\flng^*\otimes \fz(\fn)^*=\flng\otimes \fz(\fn)^*$.
%\end{equation*}
Since $\flng$ has highest weight $\vep_{i+1}+\vep_{i+2}$,
this shows that the highest weight $\nu$ for $V(\mu+\geng)^*$ is 
\begin{equation*}
\nu
=(\vep_{i+1}+\vep_{i+2}) - (\vep_{i-1}+\vep_i)
=-\vep_{i-1}-\vep_i+\vep_{i+1}+\vep_{i+2}.
\end{equation*}
We have
\begin{equation*}
-\vep_{i-1}-\vep_i+\vep_{i+1}+\vep_{i+2}
=-2(\vep_i-\vep_{i+1})-(\vep_{i-1}-\vep_i) - (\vep_{i+1}-\vep_{i+2})
\end{equation*}
with $\vep_i-\vep_{i+1} \in \gD(\fg(1))$ and 
$\vep_{i-1}-\vep_i$, $\vep_{i+1}-\vep_{i+2} \in \Pi(\fl)$
(see Appendix \ref{chap:Data}). 
Now we claim that $(\vep_{i-1}-\vep_i, \vep_i-\vep_{i+1})$
links $-(\vep_{i+1}-\vep_{i+2})-\gl_i+\rho$ to 
$-2(\vep_i-\vep_{i+1})-(\vep_{i-1}-\vep_i)-(\vep_{i+1}-\vep_{i+2})-\gl_i+\rho$.
This is to show that 
\begin{equation*}
s_{\vep_i-\vep_{i+1}}s_{\vep_{i-1}-\vep_i}(-(\vep_{i+1}-\vep_{i+2})-\gl_i+\rho)
=-2(\vep_i-\vep_{i+1})-(\vep_{i-1}-\vep_i)-(\vep_{i+1}-\vep_{i+2})-\gl_i+\rho
\end{equation*}
with
\begin{equation*}
\IP{-(\vep_{i+1}-\vep_{i+2})-\gl_i+\rho}{(\vep_{i-1}-\vep_i)^\vee} \in \Z_{\geq 0}
\end{equation*}
and
\begin{equation*}
\IP{s_{\vep_{i-1}-\vep_i}(-(\vep_{i+1}-\vep_{i+2})-\gl_i+\rho)}{(\vep_i-\vep_{i+1})^\vee}
\in \Z_{\geq 0}.
\end{equation*}
(See Definition \ref{Def:Link}.)
As $\vep_{i-1}-\vep_i \in \Pi(\fl)$, we have 
$\IP{\gl_i}{(\vep_{i-1}-\vep_i)^\vee} = 0$. Since
$\IP{\rho}{(\vep_{i-1}-\vep_i)^\vee} = 1$, it follows that
\begin{equation*}
\IP{-(\vep_{i+1}-\vep_{i+2})-\gl_i+\rho}{(\vep_{i-1}-\vep_i)^\vee} = 1 \in \Z_{\geq 0}.
\end{equation*}
Thus,
\begin{equation*}
s_{\vep_{i-1}-\vep_i}(-(\vep_{i+1}-\vep_{i+2})-\gl_i+\rho)
=-(\vep_{i-1}-\vep_i)-(\vep_{i+1}-\vep_{i+2})-\gl_i+\rho.
\end{equation*}
Next, as $\vep_i-\vep_{i+1}$ is the simple root that determines the parabolic $\fq$,
we have $\IP{\gl_i}{(\vep_i-\vep_{i+1})^\vee} = 1$. Since 
$\IP{\rho}{(\vep_i-\vep_{i+1})^\vee} = 1$, it follows that 
\begin{align*}
&\IP{s_{\vep_{i-1}-\vep_i}(-(\vep_{i+1}-\vep_{i+2})-\gl_i+\rho)}{(\vep_i-\vep_{i+1})^\vee}\\
&=\IP{-(\vep_{i-1}-\vep_i)-(\vep_{i+1}-\vep_{i+2})-\gl_i+\rho}{(\vep_i-\vep_{i+1})^\vee}\\
&=2 \in \Z_{\geq 0}.
\end{align*}
Therefore,
\begin{align*}
&s_{\vep_i-\vep_{i+1}}s_{\vep_{i-1}-\vep_i}(-(\vep_{i+1}-\vep_{i+2})-\gl_i+\rho)\\
&=s_{\vep_i-\vep_{i+1}}(-(\vep_{i-1}-\vep_i)-(\vep_{i+1}-\vep_{i+2})-\gl_i+\rho)\\
&=-2(\vep_i-\vep_{i+1})-(\vep_{i-1}-\vep_i)-(\vep_{i+1}-\vep_{i+2})-\gl_i+\rho.
\end{align*}

\vskip 0.1in

2. \underline{$V(\mu+ \geg)$ for $D_n(i)$
for $3 \leq i \leq n-3$:}
Since, by Table \ref{T93}, the special value $s_0$ is $s_0 = n-i-1$,
we want to show that there is $\ga \in \Pi(\fl)$ so that 
$-\ga -(n-i-1)\gl_i + \rho$ is linked to $\nu-(n-i-1)\gl_i + \rho$.
By Table \ref{T51}, we have $\mu+\geg = 2\vep_1$.
Observe that if $\ga_j = \vep_j-\vep_{j+1}$ 
and $w_j = s_{\ga_1}s_{\ga_2}\cdots s_{\ga_j}$ for $1 \leq j \leq i-1$  
then the longest element $w_0$ of the Weyl group of type $A_{i-1}$
may be expressed as $w_0 = w_{i-1}w_{i-2}\cdots w_{1}$.
It is shown in Section 6 of \cite{KuboThesis1} that 
$V(\mu+\geg)$ is an $\flg$-submodule of $\flg\otimes \fz(\fn)$.
Since $\flg$ is of type $A_{i-1}$ (see Appendix \ref{chap:Data}), 
the highest weight $\nu$ for 
$V(\mu+\geg)^*$ is then given by
\begin{equation*}
\nu = -w_0(2\vep_1) = -2\vep_i.
\end{equation*}
We have
\begin{equation*}
-2\vep_i = -2(\vep_i-\vep_{n-1}) - (\vep_{n-1} - \vep_{n}) \
- (\vep_{n-1} + \vep_{n})
\end{equation*}
with $\vep_i-\vep_{n-1} \in \gD(\fg(1))$ and $\vep_{n-1} - \vep_{n}$,
$\vep_{n-1} + \vep_{n} \in \Pi(\fl)$.
Then a direct computation shows that 
$(\vep_{n-1} - \vep_{n}, \vep_i-\vep_{n-1})$ links 
$-(\vep_{n-1} + \vep_{n})-(n-i-1)\gl_i + \rho$ to 
$-2\vep_i-(n-i-1)\gl_i + \rho$.

\vskip 0.1in

\underline{3. $V(\mu+\geg)$ for $E_6(3)$:}
Since, by Table \ref{T93}, the special value $s_0$ is $s_0 = 1$,
we want to show that there is $\ga \in \Pi(\fl)$ so that 
$-\ga -\gl_3 + \rho$ is linked to $\nu-\gl_3 + \rho$.
By Table \ref{T52}, we have 
\begin{equation*}
\mu+\geg 
=\ga_1+2\ga_2+2\ga_3+4\ga_4+3\ga_5+2\ga_6.
\end{equation*}
As $V(\mu+\geg)$ is a simple $\flg$-submodule of $\flg\otimes \fz(\fn)$,
if $w_0$ is the longest element of the Weyl group of $\flg$ then,
by using \textsf{LiE}, the highest weight $\nu$ for $V(\mu+\geg)^*$ is
given by
\begin{align*}
\nu 
&= -w_0(\ga_1+2\ga_2+2\ga_3+4\ga_4+3\ga_5+2\ga_6)\\
&=-2\ga_3-\ga_1-\ga_4
\end{align*}
with $\ga_3 \in \gD(\fg(1))$ and $\ga_1$, $\ga_4 \in \Pi(\fl)$.
Now a direct computation shows that 
$(\ga_1, \ga_3)$ links $-\ga_4-\gl_3+\rho$ to 
$(-2\ga_3-\ga_1-\ga_4)-\gl_3+\rho$.
\end{proof}

%%%%%%%%%%%%%%%%%%%%%%%%%%%%%%%%%%%%%%%%%
\subsection{The $V(\mu+\geg)$ case for $B_n(i)$ for $3\leq i \leq n-1$} 

Now we consider the case 
$V(\mu+\geg)$ for $B_n(i)$ for $3\leq i \leq n-1$.
By Table \ref{T93}, the special value $s_0$ is
$s_0 = n-i-(1/2)$ for $1\leq i \leq n-1$. 
(Note that when $i=n-1$, we have $s_0 = 1/2 = n-(n-1)-(1/2)$).
By the same argument used for the case $V(\mu+\geg)$ 
of $D_n(i)$ in the proof of Theorem \ref{Thm:MapO2Int},
the highest weight $\nu$ for $V(\mu+\geg)^*$ is
$\nu = -2\vep_i$.
Therefore, we have
\begin{equation}\label{Eqn:Form15}
\varphi_{\Omega_2}: 
M_\fq( -2\vep_i-(n-i-(1/2))\gl_i+\rho) 
\to M_\fq (-(n-i-(1/2))\gl_i+\rho).
\end{equation}

We first show that the standard map $\varphi_{std}$ is non-zero.
If $\gb = \sum_{\ga \in \Pi}m_{\ga} \ga \in 
\sum_{\ga \in \Pi}\Z \ga$
then we say that $|m_{\ga}|$ are 
the \emph{multiplicities} of $\ga$ in $\gb$.

\begin{Prop}\label{Prop:Prop1}
If $\fq$ is the maximal parabolic subalgebra
of type $B_n(i)$ with $3 \leq i \leq n-1$ then
the standard map $\varphi_{std}$ from
$M_\fq( -2\vep_i-(n-i-(1/2))\gl_i+\rho)$ 
to $M_\fq( -(n-i-(1/2))\gl_i+\rho)$
is non-zero.
\end{Prop}

\begin{proof}
First note that, as $s_0 = n-i-(1/2) \notin \Z$, 
Proposition \ref{Prop:Std2} cannot be applied to this case.
Then, to prove this proposition, we observe Proposition \ref{Prop:Std}; 
we show that there is no $\ga \in \Pi(\fl)$ so that 
$-\ga-(n-i-(1/2))\gl_i+\rho$ is linked to $-2\vep_i-(n-i-(1/2))\gl_i+\rho$. 
For simplicity we write
\begin{equation*}
\gd(i) = -(n-i-(1/2))\gl_i + \rho.
\end{equation*}
Since $\vep_i = \sum_{j=i}^n\ga_j$
with $\ga_j$ simple roots in the standard numbering,
we want to show that there is no $\ga \in \Pi(\fl)$ so that 
$-\ga +\gd(i)$ is linked to $-2\vep_i+\gd(i) = -2\sum_{j=i}^n\ga_j+\gd(i)$.
Suppose that such $\ga' \in \Pi(\fl)$ exists. 
Let $(\gb_1, \ldots, \gb_m)$ be a link from $-\ga' + \gd(i)$ to 
$-2\sum_{j=i}^n\ga_j+\gd(i)$. Without loss of generality,
we assume that for all $j = 1,\ldots, m$,
\begin{equation*}\label{Eqn:Assmp}
\IP{s_{\gb_{j-1}}\cdots s_{\gb_1}(-\ga'+\gd(i))}{\cgb_j} \neq 0.
\end{equation*}
(If $j=1$ then set $s_{\gb_0} = e$, the identity.) 
By the property (2) in Definition \ref{Def:Link}, this means that 
we assume that
\begin{equation}\label{Eqn:Assmp}
\IP{s_{\gb_{j-1}}\cdots s_{\gb_1}(-\ga'+\gd(i))}{\cgb_j} \in 1+\Z_{\geq 0}
\end{equation}
for all $j = 1,\ldots, m$.
Observe that it follows from the property (2) in Definition \ref{Def:Link}
that any weight linked from $-\ga'+\gd(i)$ is of the from
\begin{equation}\label{Eqn:Form}
(-\sum_{\ga \in \Pi}n_\ga \ga) -\ga'+\gd(i) \;
\text{with $n_\ga \in \Z_{\geq 0}$. }
\end{equation} 
We have $\gD^+ = \gD^+(\fl)\cup \gD(\fg(1)) \cup \gD(\fz(\fn))$,
where $\gD^+(\fl)$, $\gD(\fg(1))$, and $\gD(\fz(\fn))$ are the sets of
the positive roots in which $\ga_i$ has multiplicity zero, one, and two, respectively. 
As $(\gb_1, \ldots, \gb_m)$ is a link from $-\ga'+\gd(i)$
to $-2\sum_{j=i}^n\ga_j+\gd(i)$, we have
\begin{equation}\label{Eqn:Form2}
s_{\gb_m}\cdots s_{\gb_1}(-\ga'+\gd(i))=-2\sum_{j=i}^n\ga_j+\gd(i).
\end{equation}
If $\gb_j \in \gD^+(\fl)$ for all $j$ then we would have
\begin{equation*} 
-2\sum_{j=i}^n\ga_j+\gd(i)
=s_{\gb_m}\cdots s_{\gb_1}(-\ga'+\gd(i)) \nonumber \\
=(-\sum_{\ga \in \Pi(\fl)}k_\ga \ga) -\ga'+\gd(i)
\end{equation*}
for some $k_\ga \in \Z_{\geq 0}$. This implies that
\begin{equation}\label{Eqn:Form3}
-2\ga_i - 2\sum_{j=i+1}^n\ga_j
=(-\sum_{\ga \in \Pi(\fl)}k_\ga \ga) -\ga'.
\end{equation}
This is absurd, because,
as $\Pi(\fl) = \Pi\backslash \{\ga_i\}$ and $\ga' \in \Pi(\fl)$,
the simple root $\ga_i$ does not contribute to the right hand side of (\ref{Eqn:Form3}).
Thus, there must exist at least one $\gb_j$ in $(\gb_1, \ldots, \gb_m)$
with $\gb_j \in \gD(\fg(1))\cup\gD(\fz(\fn))$.

Now we show that any $\gb_j$ in 
$(\gb_1, \ldots, \gb_m)$ cannot belong to  $\gD(\fg(1)) \cup \gD(\fz(\fn))$. 
First, suppose that there exists $\gb_r$ in 
$(\gb_1, \ldots, \gb_m)$ with $\gb_r \in \gD(\fz(\fn))$. 
Observe that $\gD(\fz(\fn))$ consists 
of the positive roots $\vep_j+\vep_k$ for $1 \leq j < k \leq i$
(see Appendix \ref{chap:Data}). So $\gb_r$ is 
$\gb_r = \vep_s + \vep_t$ for some $1 \leq s < t \leq i$.
Since each $\vep_l = \sum_{j=l}^n \ga_j$
with $\ga_j$ simple roots,
the positive root $\gb_r = \vep_s + \vep_t$ with
$1 \leq s < t \leq i$ can be expressed as
\begin{equation*}
\gb_r = \vep_s + \vep_t
=\sum_{j=s}^{t-1}\ga_j + 2\sum_{j=t}^{n}\ga_j.
\end{equation*}
If $c = \IP{s_{\gb_{r-1}}\cdots s_{\gb_1}(-\ga'+\gd(i))}{\cgb_r}$ then 
\begin{align}\label{Eqn:Form6}
s_{\gb_r}\cdots s_{\gb_1}(-\ga'+\gd(i))
&=s_{\gb_{r-1}}\cdots s_{\gb_1}(-\ga'+\gd(i))-c\gb_r \nonumber\\
&=s_{\gb_{r-1}}\cdots s_{\gb_1}(-\ga'+\gd(i))
-c \big( \sum_{j=s}^{t-1}\ga_j + 2\sum_{j=t}^n\ga_j \big).
\end{align}
Observe that, by (\ref{Eqn:Form}), 
$s_{\gb_{r-1}}\cdots s_{\gb_1}(-\ga'+\gd(i))$ is of the form
\begin{equation}\label{Eqn:Form7}
s_{\gb_{r-1}}\cdots s_{\gb_1}(-\ga'+\gd(i))
=(-\sum_{\ga \in \Pi}m_\ga \ga ) - \ga' + \gd(i)
\end{equation}
for some $m_\ga \in \Z_{\geq 0}$.
Moreover, as $s_{\gb_m}\cdots s_{\gb_1}(-\ga'+\gd(i))$ is a weight linked from
$s_{\gb_{r}}\cdots s_{\gb_1}(-\ga'+\gd(i))$, 
the weight $s_{\gb_m}\cdots s_{\gb_1}(-\ga'+\gd(i))$ is of the form
\begin{equation}\label{Eqn:Form8}
s_{\gb_m}\cdots s_{\gb_1}(-\ga'+\gd(i))
=(-\sum_{\ga \in \Pi}m_\ga' \ga )+s_{\gb_{r}}\cdots s_{\gb_1}(-\ga'+\gd(i))
\end{equation}
for some $m_\ga' \in \Z_{\geq 0}$.
By combining (\ref{Eqn:Form6}), (\ref{Eqn:Form7}), 
and (\ref{Eqn:Form8}), we have
\begin{align}\label{Eqn:Form9}
&s_{\gb_m}\cdots s_{\gb_1}(-\ga'+\gd(i))\nonumber\\
&=(-\sum_{\ga \in \Pi}m_\ga' \ga )+s_{\gb_{r}}\cdots s_{\gb_1}(-\ga'+\gd(i))\nonumber\\
&=(-\sum_{\ga \in \Pi}m_\ga' \ga ) + s_{\gb_{r-1}}\cdots s_{\gb_1}(-\ga'+\gd(i))
-c \big( \sum_{j=s}^{t-1}\ga_j + 2\sum_{j=t}^n\ga_j \big)\nonumber\\
&=(-\sum_{\ga \in \Pi}m_\ga' \ga ) + (-\sum_{\ga \in \Pi}m_\ga \ga )
-c \big( \sum_{j=s}^{t-1}\ga_j + 2\sum_{j=t}^n\ga_j \big)- \ga' + \gd(i)
\end{align}
with $m_\ga, m_\ga' \in \Z_{\geq 0}$.
By (\ref{Eqn:Assmp}), we have
\begin{equation*}
c = \IP{s_{\gb_{r-1}}\cdots s_{\gb_1}(-\ga'+\gd(i))}{\cgb_r} \in 1+\Z_{\geq 0}.
\end{equation*}
Therefore,
by (\ref{Eqn:Form9}), the weight $s_{\gb_m}\cdots s_{\gb_1}(-\ga'+\gd(i))$ 
is of the form
\begin{equation*}
s_{\gb_m}\cdots s_{\gb_1}(-\ga'+\gd(i))
=-\sum_{\ga \in \Pi}n_\ga \ga 
-\sum_{j=s}^{t-1}\ga_j - 2\sum_{j=t}^{n}\ga_j
-\ga'+\gd(i)
\end{equation*}
for some $n_\ga \in \Z_{\geq 0}$.
By (\ref{Eqn:Form2}), this implies that
\begin{equation*}
2\sum_{j=i}^n\ga_j
=\sum_{\ga \in \Pi}n_\ga \ga 
+\sum_{j=s}^{t-1}\ga_j + 2\sum_{j=t}^{n}\ga_j
+\ga'.
\end{equation*}
Since $s<t \leq i$, we then have
\begin{align}\label{Eqn:Form5}
0= &\sum_{\ga \in \Pi}n_\ga \ga 
+\sum_{j=s}^{t-1}\ga_j + 2\sum_{j=t}^{n}\ga_j
+\ga' - 2\sum_{j=i}^n\ga_j \nonumber\\
&=
\begin{cases}
\sum_{\ga \in \Pi}n_\ga \ga 
+\sum_{j=s}^{t-1}\ga_j + 2\sum_{j=t}^{i-1}\ga_j
+\ga'  & \text{ if $t < i$ }\\
\sum_{\ga \in \Pi}n_\ga \ga 
+\sum_{j=s}^{t-1}\ga_j 
+\ga' & \text{ if $t = i$. }
\end{cases}
\end{align}
This is a contradiction, because, as $n_\ga \in \Z_{\geq 0}$,
(\ref{Eqn:Form5}) cannot be zero.
Therefore no $\gb_j$ in $(\gb_1, \ldots, \gb_m)$ 
is a root in $\gD(\fz(\fn))$. 

Next we suppose that there exists 
$\gb_r$ in $(\gb_1, \ldots, \gb_m)$ with $\gb_r \in \gD(\fg(1))$. 
There are long roots and short roots in $\gD(\fg(1))$.
We handle these cases separately. 
We first suppose that $\gb_r$ is a long root in $\gD(\fg(1))$.
The long roots in $\gD(\fg(1))$ are 
$\vep_j \pm \vep_k$ for $1\leq j\leq i$ and $i+1 \leq k \leq n$.
(See Appendix \ref{chap:Data}.) The roots $\vep_j \pm \vep_k$
may be expressed in terms of simple roots as 
\begin{equation*}
\vep_j + \vep_k
=\sum_{l=j}^n \ga_l + \sum_{l=k}^n\ga_l
=\sum_{l=j}^{i-1} \ga_l + \ga_i +  \sum_{l=i+1}^{k-1}\ga_l
+2\sum_{l=k}^{n-1}\ga_l + 2\ga_n
\end{equation*}
and
\begin{equation*}
\vep_j - \vep_k
=\sum_{l=j}^n \ga_l - \sum_{l=k}^n\ga_l
=\sum_{l=j}^{i-1} \ga_l + \ga_i +  \sum_{l=i+1}^{k-1}\ga_l.
\end{equation*}
We show that if $\gb_r = \vep_j\pm\vep_k$ then 
$\IP{s_{\gb_{r-1}}\cdots s_{\gb_1}(-\ga'+\gd(i))}{\cgb_r} \notin \Z$.
Observe that since $\ga_n$ is the only short simple root, the coroot
$(\vep_j+\vep_k)^\vee$ can be expressed as
\begin{align*}
&(\vep_j+\vep_k)^\vee\\
&= \big( \sum_{l=j}^{i-1} \ga_l + \ga_i +  \sum_{l=i+1}^{k-1}\ga_l
+2\sum_{l=k}^{n-1}\ga_l + 2\ga_n \big)^\vee\\
&=\sum_{l=j}^{i-1} \frac{2\ga_l}{||\vep_j+\vep_k||^2}
+ \frac{2\ga_i}{||\vep_j+\vep_k||^2}
+ \sum_{l=i+1}^{k-1} \frac{2\ga_l}{||\vep_j+\vep_k||^2}
+2\sum_{l=k}^{n-1}\frac{2\ga_l}{||\vep_j+\vep_k||^2}
+ 2 \cdot \frac{2\ga_n}{||\vep_j+\vep_k||^2}\\
&=\sum_{l=j}^{i-1} \cga_l + \cga_i
+ \sum_{l=i+1}^{k-1} \cga_l
+2\sum_{l=k}^{n-1}\cga_l
+ \cga_n.
\end{align*}
Similarly, we have
\begin{equation*}
(\vep_j-\vep_k)^\vee
=\sum_{l=j}^{i-1} \cga_l + \cga_i
+ \sum_{l=i+1}^{k-1} \cga_l.
\end{equation*}
Now observe that, as $\gl_i$ is the fundamental weight for $\ga_i$, for $\ga \in \Pi$, 
we have
\begin{align}\label{Eqn:Link}
\IP{\gd(i)}{\cga}
&=\IP{-(n-i-(1/2))\gl_i+\rho}{\cga}\nonumber\\
&=
\begin{cases}
-n+i+(3/2)  & \text{if $\ga=\ga_i$}\\
1 & \text{otherwise.}
\end{cases}
\end{align}
Thus,
\begin{align*}
&\IP{\gd(i)}{(\vep_j+\vep_k)^\vee}\\
&=\IP{\gd(i)}{\sum_{l=j}^{i-1} \cga_l + \cga_i
+ \sum_{l=i+1}^{k-1} \cga_l
+2\sum_{l=k}^{n-1}\cga_l
+ \cga_n}\\
&=\sum_{l=j}^{i-1} \IP{\gd(i)}{\cga_l} + 
\IP{\gd(i)}{\cga_i}
+ \sum_{l=i+1}^{k-1} \IP{\gd(i)}{\cga_l}
+2\sum_{l=k}^{n-1} \IP{\gd(i)}{\cga_l}
+ \IP{\gd(i)}{\cga_n}\\
&=(i-1 - (j-1)) + (-n+i+(3/2)) +(k-1 -i)
+2(n-1 - (k-1)) + 1\\
&=n-k+i-j+(3/2).
\end{align*}
Similarly,
\begin{equation*}
\IP{\gd(i)}{(\vep_j-\vep_k)^\vee}
=-n+k+i-j+(1/2).
\end{equation*}
Hence, for $\gb_r  = \vep_j \pm \vep_k$, we have
$\IP{\gd(i)}{\cgb_r} \notin \Z$. Now, by (\ref{Eqn:Form7}), we have 
\begin{align*}
\IP{s_{\gb_{r-1}}\cdots s_{\gb_1}(-\ga'+\gd(i))}{\cgb_r}
&=\IP{(-\sum_{\ga \in \Pi}m_\ga \ga ) - \ga' + \gd(i)}{\cgb_r}\\
&=-\sum_{\ga \in \Pi}m_\ga \IP{\ga}{\cgb_r} - \IP{\ga'}{\cgb_r}
+\IP{\gd(i)}{\cgb_r}
\end{align*}
with $m_\ga \in \Z$. Since $m_\ga, \IP{\ga}{\cgb_r}, \IP{\ga'}{\cgb_r} \in \Z$
and $\IP{\gd(i)}{\cgb_r} \notin \Z$, this shows that 
$\IP{s_{\gb_{r-1}}\cdots s_{\gb_1}(-\ga'+\gd(i))}{\cgb_r}\notin \Z$.

Finally, we suppose that $\gb_r$ is a short root in $\gD(\fg(1))$.
The short roots in $\gD(\fg(1))$
are $\vep_j$ for $1\leq j \leq i$ (see Appendix \ref{chap:Data}). 
Thus $\gb_r$ is $\gb_r = \vep_l$ for some $1\leq l \leq i$.
Since $\vep_l$ is of the form 
$\vep_l = \sum_{j=l}^n\ga_j$,
(\ref{Eqn:Form2}) forces that $l = i$; otherwise,
$s_{\gb_m}\cdots s_{\gb_1}(-\ga'+\gd(i))$ would have a contribution from 
some $\ga_j \in \Pi$ with $1\leq j \leq i-1$.
Thus $\gb_r = \vep_i = \sum_{j=i}^n\ga_j$.
%Observe that, as $3\leq i \leq n-1$, the simple root $\ga_i$ is a long root.
Since $\gb_r$ is a short root,
the coroot $\cgb_r = (\sum_{j=i}^n\ga_j)^\vee$ 
can be expressed as
\begin{equation*}
\cgb_r 
= \big( \sum_{j=i}^n\ga_j \big)^\vee
= \sum_{j=i}^n\frac{2\ga_j}{||\gb_r||^2}
=\frac{2\ga_i}{||\gb_r||^2}
+\sum_{j=i+1}^{n-1}\frac{2\ga_j}{||\gb_r||^2}
+\frac{2\ga_n}{||\gb_r||^2}
=2\cga_i + 2\sum_{j=i+1}^{n-1}\cga_j + \cga_n.
\end{equation*}
It then follows from (\ref{Eqn:Link}) that 
\begin{align*}%\label{Eqn:Form14}
\IP{\gd(i)}{\cgb_r}
&=\IP{-(n-i-(1/2))\gl_i+\rho}{\big( \sum_{j=i}^n\ga_j \big)^\vee}\nonumber\\
&=\IP{-(n-i-(1/2))\gl_i+\rho}{2\cga_i + 2\sum_{j=i+1}^{n-1}\cga_j + \cga_n} \nonumber\\
&=2\IP{-(n-i-(1/2))\gl_i+\rho}{\cga_i}
+2\sum_{j=i+1}^{n-1}\IP{-(n-i-(1/2))\gl_i+\rho}{\cga_j} \nonumber \\
&\quad +\IP{-(n-i-(1/2))\gl_i+\rho}{\cga_n}\\
&=2(-n+i+(3/2))+2(n-1-i)+1\\
&=2.
\end{align*}
Thus, by (\ref{Eqn:Form7}), we have
\begin{align}\label{Eqn:Form14}
\IP{s_{\gb_{r-1}} \cdots s_{\gb_1}(-\ga'+\gd(i))}{\cgb_r}
&=\IP{(-\sum_{\ga \in \Pi}m_\ga \ga ) - \ga' + \gd(i)}{\cgb_r} \nonumber\\
&=\IP{-\sum_{\ga \in \Pi}m_\ga \ga - \ga'}{\cgb_r} + 2
\end{align}
with $m_\ga \in \Z_{\geq 0}$.
Thus, as $\gb_r = \sum_{j=i}^n\ga_j$, 
if $d = \IP{-\sum_{\ga \in \Pi}m_\ga \ga - \ga'}{\cgb_r} + 2$
then $s_{\gb_{r}} \cdots s_{\gb_1}(-\ga'+\gd(i))$ is of the form
\begin{equation*}
s_{\gb_{r}} \cdots s_{\gb_1}(-\ga'+\gd(i))
=s_{\gb_{r-1}} \cdots s_{\gb_1}(-\ga'+\gd(i))-d\sum_{j=i}^n\ga_j.
\end{equation*}
By (\ref{Eqn:Form7}) and (\ref{Eqn:Form8}), we have
\begin{align*}
s_{\gb_m}\cdots s_{\gb_1}(-\ga'+\gd(i))
&=(-\sum_{\ga \in \Pi}m_\ga' \ga )+s_{\gb_{r}}\cdots s_{\gb_1}(-\ga'+\gd(i))\nonumber \\
&=(-\sum_{\ga \in \Pi}m_\ga' \ga )+ s_{\gb_{r-1}} \cdots s_{\gb_1}(-\ga'+\gd(i))
-d\sum_{j=i}^n\ga_j\nonumber \\
&=(-\sum_{\ga \in \Pi}m_\ga' \ga )+(-\sum_{\ga \in \Pi}m_\ga \ga ) -d\sum_{j=i}^n\ga_j- \ga' + \gd(i)
\end{align*}
with $m_\ga, m_\ga' \in \Z_{\geq}$. Therefore,
$s_{\gb_m}\cdots s_{\gb_1}(-\ga'+\gd(i))$ can be expressed as
\begin{equation*}
s_{\gb_m}\cdots s_{\gb_1}(-\ga'+\gd(i))
=-\sum_{\ga \in \Pi}n_\ga \ga -d\sum_{j=i}^n\ga_j- \ga' + \gd(i)
\end{equation*}
for some $n_\ga \in \Z_{\geq 0}$.
By (\ref{Eqn:Form2}), this implies that
\begin{equation}\label{Eqn:Form10}
2\sum_{j=i}^n\ga_j
=\sum_{\ga \in \Pi}n_\ga \ga + d\sum_{j=i}^n\ga_j+ \ga'.
\end{equation}
By comparing the coefficients of $\ga_i$ in the both sides, we have 
\begin{equation}\label{Eqn:Form11}
n_{\ga_i}+d = 2.
\end{equation}
By (\ref{Eqn:Assmp}) and (\ref{Eqn:Form14}), we have 
$d = \IP{-\sum_{\ga \in \Pi}m_\ga \ga - \ga'}{\cgb_r} + 2 \in 1+\Z_{\geq 0}$.
Since $n_{\ga_i} \in \Z_{\geq 0}$, (\ref{Eqn:Form11}) forces that
\begin{equation*}
d = 2 \text{ or } d=1.
\end{equation*}
If $d=2$ then (\ref{Eqn:Form10}) becomes 
\begin{equation*}
2\sum_{j=i}^n\ga_j
=\sum_{\ga \in \Pi}n_\ga \ga + 2\sum_{j=i}^n\ga_j+ \ga'.
\end{equation*}
Therefore,
\begin{equation}\label{Eqn:Form4}
\sum_{\ga \in \Pi}n_\ga \ga +\ga' = 0,
\end{equation}
which is a contradiction, because as $\ga' \in \Pi$ and $k'_\ga \in \Z_{\geq 0}$, 
the left hand side of (\ref{Eqn:Form4}) cannot be zero.
If $d=1$ then, since $d = \IP{-\sum_{\ga \in \Pi}m_\ga \ga - \ga'}{\cgb_r} + 2$, 
we have
\begin{equation*}
\IP{-\sum_{\ga \in \Pi}m_\ga \ga - \ga'}{\cgb_r} + 2 = 1.
\end{equation*}
Thus,
\begin{equation}\label{Eqn:Form12}
\IP{\sum_{\ga \in \Pi}m_\ga \ga + \ga'}{\cgb_r} = 1.
\end{equation}
Observe that, as $\gb_r = \vep_i$ in the standard realization, 
if $\IP{\ga}{\cgb_r} \neq 0$ for $\ga \in \Pi$ then $\ga$ must be
$\ga = \vep_{i-1}-\vep_{i}$ in $\Pi(\fl)$ or $\ga = \vep_i-\vep_{i+1}$ 
in $\Pi\backslash \Pi(\fl)$. 
Since $\IP{\vep_{i-1}-\vep_i}{\vep_i^\vee} = -2$,
$\IP{\vep_{i}-\vep_{i+1}}{\vep_i^\vee}=2$, and $\ga' \in \Pi(\fl)$,
the left hand side of (\ref{Eqn:Form12}) is 
\begin{align*}
\IP{\sum_{\ga \in \Pi}m_\ga \ga + \ga'}{\cgb_r}
&= m_{\vep_{i-1}-\vep_i}\IP{\vep_{i-1}-\vep_i}{\vep_i^\vee}
+m_{\vep_{i}-\vep_{i+1}}\IP{\vep_{i}-\vep_{i+1}}{\vep_i^\vee}
+\IP{\ga'}{\vep_i^\vee}\\
&=-2m_{\vep_{i-1}-\vep_i} + 2m_{\vep_{i}-\vep_{i+1}}
-2\gd_{\ga', \vep_{i-1}-\vep_{i}}\\
&=2(m_{\vep_{i}-\vep_{i+1}}-m_{\vep_{i-1}-\vep_i}- \gd_{\ga', \vep_{i-1}-\vep_{i}}),
\end{align*}
where $\gd_{\ga', \vep_{i-1}-\vep_{i}}$ is the Kronecker delta.
As $m_{\vep_{i}-\vep_{i+1}}$, $m_{\vep_{i-1}-\vep_i}$, 
and $\gd_{\ga', \vep_{i-1}-\vep_{i}}$ are integers, this shows that 
$\IP{\sum_{\ga \in \Pi}m_\ga \ga + \ga'}{\cgb_r} \neq 1$, which 
contradicts (\ref{Eqn:Form12}).
Therefore, no $\gb_r$ in $(\gb_1, \ldots, \gb_m)$ is a short root 
in $\gD(\fg(1))$.
Hence there is no link from
$-\ga'+\gd(i)$ to $-2\sum_{j=i}^n\ga_j+\gd(i)$.
\end{proof}

\vsp

Now we are going to show that 
the map 
\begin{equation*}%\label{Eqn:Form15}
\varphi_{\Omega_2}: 
M_\fq( -2\vep_i-(n-i-(1/2))\gl_i+\rho) 
\to M_\fq (-(n-i-(1/2))\gl_i+\rho)
\end{equation*}
is standard. 
This is to show that,
given highest weight vector $v_h$ for $F(\Omega_2|_{V(\mu+\geg)^*})$,
the image $\varphi_{\Omega_2}(1\otimes v_h)$ of $1\otimes v_h$
is a non-zero scalar multiple of $\varphi_{std}(1\otimes v_h)$, where
$F(\Omega_2|_{V(\mu+\geg)^*})$ is
the finite dimensional simple $\fl$-submodule of 
$M_\fq(-(n-i-(1/2))\gl_i+\rho)^\fn$ 
induced by the $\Omega_2|_{V(\mu+\geg)^*}$ system,
so that
$M_\fq( -2\vep_i-(n-i-(1/2))\gl_i+\rho) 
= \Cal{U}(\fg)\otimes_{\Cal{U}(\fq)} F(\Omega_2|_{V(\mu+\geg)^*})$.

Observe that, by the definition of $\varphi_{\Omega_k}$, we have
$\varphi_{\Omega_2}(1\otimes v_h) = 1 \cdot v_h = v_h$.
On the other hand, if $1\otimes v^+$ 
is a highest weight vector for $M(-2\vep_i-(n-i-(1/2))\gl_i+\rho)$
with highest weight $-2\vep_i-(n-i-(1/2))\gl_i$
and if
$\text{pr}: M(-(n-i-(1/2))\gl_i+\rho) \to M_\fq(-(n-i-(1/2))\gl_i+\rho)$
is the canonical projection map then $\varphi_{std}(1\otimes v_h) = 
(\text{pr}\circ \varphi)(1\otimes v^+)$, where $\varphi$ is an embedding
of $M(-2\vep_i-(n-i-(1/2))\gl_i+\rho)$ into $M(-(n-i-(1/2))\gl_i+\rho)$; 
in a diagram we have
\begin{equation*}
\xymatrix{ 
M(-2\vep_i-(n-i-(1/2))\gl_i+\rho) \ar[r]^\varphi 
\ar[d]_{\text{pr}'} &
M(-(n-i-(1/2))\gl_i+\rho) 
\ar[d]^{\text{pr}}\\
M_\fq( -2\vep_i-(n-i-(1/2))\gl_i+\rho) \ar[r]^{\varphi_{std}}&
M_\fq(-(n-i-(1/2))\gl_i+\rho),
}
\end{equation*}
where $\text{pr}': M(-2\vep_i-(n-i-(1/2))\gl_i+\rho) \to M_\fq( -2\vep_i-(n-i-(1/2))\gl_i+\rho)$
is the canonical projection map.
Note that, by Proposition \ref{Prop:Prop1}, we have
$(\text{pr}\circ \varphi)(1\otimes v^+) = \varphi_{std}(1\otimes v_h) \neq 0$.
Therefore, to show that $\varphi_{\Omega_2}$ is standard,
we wish to show that $v_h = \varphi_{\Omega_2}(1\otimes v_h)$ 
is a non-zero scalar multiple of $(\text{pr}\circ \varphi)(1\otimes v^+)$. 
Since $M_\fq(-(n-i-(1/2))\gl_i+\rho) \cong 
\Cal{U}(\bar\fn) \otimes \C_{-(n-i-(1/2))\gl_i}$ as an $\fl$-module,
%and since $F(\Omega_2|_{V(\mu+\geg)^*})$ is an $\fl$-submodule of 
%$M_\fq(-(n-i-(1/2))\gl_i+\rho)$, 
we have
\begin{equation}\label{Eqn:Form16}
v_h=u_h \otimes 1_{-(n-i-(1/2))\gl_i}
\end{equation}
and 
\begin{equation}\label{Eqn:Form18}
(\text{pr}\circ \varphi)(1\otimes v^+)=\tu \otimes 1_{-(n-i-(1/2))\gl_i}
\end{equation}
for some $u_h, \tilde{u} \in \Cal{U}(\bar \fn)\backslash \{0\}$. 
Hence, to show that $v_h$ is a non-zero 
scalar multiple of $(\text{pr}\circ \varphi)(1\otimes v^+)$,
it suffices to show that $u_h$ in (\ref{Eqn:Form16}) is 
a non-zero scalar multiple of $\tu$ in (\ref{Eqn:Form18}).

Observe that, as $v_h=u_h \otimes 1_{-(n-i-(1/2))\gl_i}$ 
is a highest weight vector for the simple 
$\fl$-submodule $F(\Omega_2|_{V(\mu+\geg)^*})$ of
$\Cal{U}(\bar\fn) \otimes \C_{-(n-i-(1/2))\gl_i+\rho}$, for all $\ga \in \Pi(\fl)$, 
we have $X_\ga \cdot (u_h \otimes 1_{-(n-i-(1/2))\gl_i}) = 0$.
%where $X_\ga$ acts diagonally.
Therefore $\ad(X_\ga)(u_h) = 0$ for all $\ga \in \Pi(\fl)$.
On the other hand, it follows from (\ref{Eqn:Ltype1}) that 
$F(\Omega_2|_{V(\mu+\geg)^*})$ is spanned by the elements
of the form $u \otimes 1_{-(n-i-(1/2))\gl_i}$ with $u \in \gs(\Sym^2(\bar\fn))$.
Since 
$F(\Omega_2|_{V(\mu+\geg)^*})$ has highest weight 
$-2\vep_i-(n-i-(1/2))\gl_i$, this shows that
$u_h$ is an element in $\gs(\Sym^2(\bar\fn))$ with weight $-2\vep_i$.
%where $\gs: \Sym(\bar \fn) \to \Cal{U}(\bar \fn)$ is the symmetrization map.

\begin{Def}\label{Def:Condition}
For $u \in \Cal{U}(\bar\fn)$, we say that $u$ satisfies
\emph{Condition (H)}
if $u$ satisfies following three conditions:
\begin{enumerate}
\item[(1)] $u \in \gs(\Sym^2(\bar \fn))$, 
\item[(2)] $u$ has weight $-2\vep_i$, and
\item[(3)] $\ad(X_\ga)(u) = 0$ for all $\ga \in \Pi(\fl)$.
\end{enumerate}
\end{Def}

It follows from the observation made before Definition \ref{Def:Condition} that
$u_h \in \Cal{U}(\bar \fn)$ in (\ref{Eqn:Form16}) satisfies Condition (H).
Our first goal is to show that any element in $\Cal{U}(\bar \fn)$ that 
satisfies Condition (H) is a scalar multiples of $u_h$.

\begin{Lem}\label{Lem:Lem3}
For any $\gb \in \gD^+(\fl) \cup \gD(\fz(\fn))$, we have 
$2\vep_i - \gb \notin \gD^+$.
\end{Lem}

\begin{proof}
This lemma follows from a direct observation.
(See Appendix \ref{chap:Data} for 
$\gD^+(\fl)=\gD^+(\flg)\cup \gD^+(\flng)$ and $\gD(\fz(\fn))$.)
\end{proof}

We write $\fu = \bigoplus_{\ga \in \gD^+}\fg_{\ga}$ for the nilradical 
of $\fb = \fh \oplus \fu$ and
we denote by $\bar \fu$ the opposite nilradical of $\fu$.
Note that, as $\fn$ is the nilradical of the parabolic subalgebra 
$\fq=\fl\oplus\fn$, we have $\fn \subset \fu$.

\begin{Lem}\label{Lem:Lem}
If $u$ is in $\Sym^2(\bar \fu)$ with weight $-2\vep_i$ then 
$u$ is of the form
\begin{equation*}
AX_{-\vep_i}^2 
+ \sum_{k=i+1}^nB_k X_{-(\vep_i+\vep_k)}X_{-(\vep_i-\vep_k)}
\end{equation*}
for some constants $A$ and $B_k$. In particular, we have $u \in \Sym^2(\bar \fn)$.
\end{Lem}

\begin{proof}
If $u \in \gs(\Sym^2(\bar \fu))$ with weight $-2\vep_i$ then $u$ is of the from
\begin{equation*}
u=\sum c_\gb X_{-\gb} X_{-2\vep_i+\gb}
\end{equation*}
for some constants $c_\gb$, where the sum runs 
over the roots $\gb \in \gD^+=\gD^+(\fl)\cup\gD(\fg(1)) \cup \gD(\fz(\fn))$ 
so that $2\vep_i -\gb \in \gD^+$.
By Lemma \ref{Lem:Lem3},
the roots $\gb$ must be in $\gD(\fg(1))$. 
Thus if $\gD_{2\vep_i}(\fg(1)) 
= \{ \gb \in \gD(\fg(1)) \; | \; 2\vep_i - \gb \in \gD\}$
then 
\begin{equation*}
u = \sum_{\gb \in \gD_{2\vep_i}(\fg(1))}
c_\gb X_{-\gb} X_{-2\vep_i+\gb}.
\end{equation*}
By Appendix \ref{chap:Data}, we have
\begin{equation*}
\gD(\fg(1)) = \{\vep_j \pm \vep_k \; | \; 1\leq j \leq i \text{ and }  i+1 \leq k \leq n\}
 \cup \{ \vep_j \; | \; 1 \leq j \leq i \}.
\end{equation*}
Thus,
\begin{equation*}
\gD_{2\vep_i}(\fg(1)) 
= \{ \gb \in \gD(\fg(1)) \; | \; 2\vep_i - \gb \in \gD\}\\
=\{\vep_i \pm \vep_k \; | \; i+1 \leq k \leq n\} \cup \{ \vep_i \}.
\end{equation*}
Therefore $u$ is of the form
\begin{align*}
u &= \sum_{\gb \in \gD_{2\vep_i}(\fg(1))}
c_\gb X_{-\gb} X_{-2\vep_i+\gb}\\
&=c_{\vep_i}X^2_{-\vep_i} + 
\sum_{k=i+1}^n c_{\vep_i + \vep_k}X_{-(\vep_i+\vep_k)}X_{-(\vep_i-\vep_k)}
+\sum_{k=i+1}^n c_{\vep_i - \vep_k}X_{-(\vep_i-\vep_k)}X_{-(\vep_i+\vep_k)}\\
&=c_{\vep_i}X^2_{-\vep_i} + 
\sum_{k=i+1}^n (c_{\vep_i + \vep_k}+c_{\vep_i - \vep_k})
X_{-(\vep_i+\vep_k)}X_{-(\vep_i-\vep_k)}.
\end{align*}
If $A = c_{\vep_i}$ and $B_k = c_{\vep_i + \vep_k}+c_{\vep_i - \vep_k}$
then $u$ can be expressed as 
\begin{equation*}
u = AX_{-\vep_i}^2 
+ \sum_{k=i+1}^nB_k X_{-(\vep_i+\vep_k)}X_{-(\vep_i-\vep_k)}.
\qedhere
\end{equation*}
\end{proof}

\begin{Prop}\label{Prop:Prop2}
If $u \in \Cal{U}(\bar \fn)$ satisfies Condition (H) then $u$ 
is a scalar multiple of $u_h$.
\end{Prop}

\begin{proof}
As $u_h$ satisfies Condition (H), to prove this proposition, it 
suffices to show that any element $u \in \Cal{U}(\bar\fn)$ 
that satisfies Condition (H) is a scalar multiple of 
\begin{equation}\label{Eqn:u_0}
u_0 = X^2_{-\vep_i} + 
\sum_{j=i+1}^n b_j X_{-(\vep_i+\vep_j)}X_{-(\vep_i-\vep_j)},
\end{equation}
where
\begin{equation}\label{Eqn:Constants1}
b_j = (-1)^{n-j}b_n
\prod_{k=j}^{n-1} 
\frac{N_{\vep_k-\vep_{k+1}, -(\vep_i-\vep_{k+1})}}
{N_{\vep_k-\vep_{k+1}, -(\vep_i+\vep_k)}}
\end{equation}
for $j=i+1, \ldots, n-1$ and 
\begin{equation}\label{Eqn:Constants2}
b_n = -\frac{2N_{\vep_n, -\vep_i}}{N_{\vep_n, -(\vep_i+\vep_n)}}.
\end{equation}
Here, $N_{\ga, \gb}$ are the constants so that 
$[X_\ga, X_\gb] = N_{\ga, \gb}X_{\ga+\gb}$.
%(See the beginning of Section 5.)

If $u \in \Cal{U}(\bar\fn)$ satisfies Condition (H) then 
$u \in \gs(\Sym^2(\bar \fn)) \subset \hat{\gs}(\Sym^2(\bar \fu))$ 
and has weight $-2\vep_i$, where
$\hat{\gs}: \Sym(\bar \fu) \to \Cal{U}(\bar \fu)$ is 
the symmetrization map for $\Sym(\bar \fu)$.
Thus it follows from Lemma \ref{Lem:Lem} that $u$ is of the from
\begin{equation}\label{Eqn:Form17}
u=AX_{-\vep_i}^2 
+ \sum_{k=i+1}^nB_k X_{-(\vep_i+\vep_k)}X_{-(\vep_i-\vep_k)}
\end{equation}
for some constants $A$ and $B_k$.
Now observe that, by the condition (3) in Definition \ref{Def:Condition},
we have $\ad(X_\ga)(u) = 0$ for all $\ga \in \Pi(\fl)$. 
Therefore, as $\vep_{j}-\vep_{j+1}$ and $\vep_n$ are in $\Pi(\fl)$ 
for $j =i+1, \ldots, n-1$, we have
\begin{equation*}
\ad(X_{\vep_{j}-\vep_{j+1}})(u) = 0 
\quad \text{and} \quad
\ad(X_{\vep_n})(u) = 0
\end{equation*}
for $j =i+1, \ldots, n-1$. By (\ref{Eqn:Form17}), 
this means that for $j =i+1, \ldots, n-1$,
\begin{equation*}
\ad(X_{\vep_{j}-\vep_{j+1}}) 
\big( AX_{-\vep_i}^2 
+ \sum_{k=i+1}^nB_k X_{-(\vep_i+\vep_k)}X_{-(\vep_i-\vep_k)}\big) = 0
\end{equation*}
and
\begin{equation*}
\ad(X_{\vep_n})
\big( AX_{-\vep_i}^2 
+ \sum_{k=i+1}^nB_k X_{-(\vep_i+\vep_k)}X_{-(\vep_i-\vep_k)}\big) = 0,
\end{equation*}
which are
\begin{equation*}
B_{j}\;\ad(X_{\vep_{j}-\vep_{j+1}})(X_{-(\vep_i+\vep_j)}X_{-(\vep_i-\vep_j)})
+B_{j+1}\;\ad(X_{\vep_{j}-\vep_{j+1}})
(X_{-(\vep_i+\vep_{j+1})}X_{-(\vep_i-\vep_{j+1})}) = 0
\end{equation*}
and
\begin{equation*}
A\; \ad(X_{\vep_n})(X_{-\vep_i}^2) 
+ B_n\; \ad(X_{\vep_n})(X_{-(\vep_i+\vep_n)}X_{-(\vep_i-\vep_n)}) = 0,
\end{equation*}
respectively.
By solving the system of linear equations, we obtain
$B_j = b_j A$ for $j=i+1, \ldots, n$ with $b_j$ in 
(\ref{Eqn:Constants1}) and (\ref{Eqn:Constants2}).
Therefore, by (\ref{Eqn:u_0}) and (\ref{Eqn:Form17}), 
we obtain $u = Au_0$.
\end{proof}

By Proposition \ref{Prop:Prop2}, to prove that $\varphi_{\Omega_2}$
in (\ref{Eqn:Form15}) is standard, it suffices to show that $\tu$ in 
(\ref{Eqn:Form18}) satisfies Condition (H). As  
$(\text{pr} \circ \varphi)(1\otimes v^+)=\tu \otimes 1_{-(n-i-(1/2))\gl_i}$
is a highest weight vector with highest weight $-2\vep_i-(n-i-(1/2))\gl_i$,
one can easily see that $\tu$ satisfies the conditions (2) and (3) 
in Definition \ref{Def:Condition}. So we wish to show that 
$\tu$ is in $\gs(\Sym^2(\bar \fn))$. 
To do so we need several technical lemmas.

\begin{Lem}\label{Lem:Lem1}
No polynomial in $\Sym^r(\bar \fn)$ for $r \geq 3$ has weight $-2\vep_i$.
\end{Lem}

\begin{proof}
Observe that the simple root $\ga_\fq = \ga_i$ has multiplicity $\geq 1$ in
any roots $\gb \in \gD(\fn)$. Therefore, 
in the weights for any polynomials in $\Sym^r(\bar \fn)$,
the simple root $\ga_i$ has multiplicity greater than or equal to $r$.
Since $\ga_i$ has multiplicity $2$ in $-2\vep_i = -2\sum_{j=i}\ga_j$, 
no polynomial in $\Sym^r(\bar \fn)$ for $r \geq 3$ has weight $-2\vep_i$.
\end{proof}

\begin{Cor}\label{Cor:Cor2}
Any non-zero polynomials in $\Sym^r(\bar \fu)$ with weight $-2\vep_i$ for $r\geq 3$ 
have contributions from root vectors $X_{-\ga}$ for $\ga \in \gD^+(\fl)$. 
\end{Cor}

\begin{proof}
Since $\gD(\fu) = \gD^+(\fl)\cup \gD(\fn)$,
this is an immediate consequence of Lemma \ref{Lem:Lem1}. 
\end{proof}

\begin{Lem}\label{Lem:Lem4}
If $u \in \Cal{U}(\bar \fu)$ has weight $-2\vep_i$ then 
\begin{equation}\label{Eqn:Form19}
u= AX_{-\vep_i}^2  + \sum_{k=i+1}^nB_k X_{-(\vep_i+\vep_k)}X_{-(\vep_i-\vep_k)}
+ \sum_{\ga \in \gD^+(\fl)}u^\ga X_{-\ga}
\end{equation}
for some constants $A$ and $B_k$, and some elements $u^\ga \in \Cal{U}(\bar \fu)$. 
\end{Lem}

\begin{proof}
If 
\begin{equation*}
\Cal{U}_r(\bar \fu) = \{ u \in \Cal{U}(\bar \fu) \; | \text{ $u$ has degree at most $r$}\}
\end{equation*}
then  $\Cal{U}(\bar \fu) = \bigcup_{r=1}^\infty\Cal{U}_r(\bar \fu)$ and
$\Cal{U}_{r+1}(\bar \fu)/\Cal{U}_r(\bar \fu) \cong \Sym^{r+1}(\bar \fu)$.
We show this lemma by induction on the degree $r$ for $\Cal{U}_r(\bar \fu)$.
First observe that since $-2\vep_i \notin \gD$, the element $u$ cannot be in 
$\Cal{U}_1(\bar \fu) = \C \oplus \bar \fu$. Thus if $u \in \Cal{U}_2(\bar \fu)$ then 
$u \in \Sym^2(\bar \fu) \cong \Cal{U}_2(\bar \fu)/\Cal{U}_1(\bar \fu)$. Thus, by 
Lemma \ref{Lem:Lem}, if $u \in \Cal{U}_2(\bar \fu)$ then
$u= AX_{-\vep_i}^2  + \sum_{k=i+1}^nB_k X_{-(\vep_i+\vep_k)}X_{-(\vep_i-\vep_k)}$
for some constants $A$ and $B_k$. Now assume that this lemma holds for 
$u \in \Cal{U}_r(\bar \fu)$ for $3 \leq r \leq t$, and suppose that 
$u \in \Cal{U}_{t+1}(\bar \fu)$. By Corollary \ref{Cor:Cor2}, any polynomials
in $\Cal{U}_{t+1}(\bar \fu)/\Cal{U}_t(\bar \fu) \cong \Sym^{t+1}(\bar \fu)$ with
weight $-2\vep_i$ have contributions from root vectors in $\fl$. By permuting
the root vectors, in $\Cal{U}_{t+1}(\bar\fu)$, those polynomials can be expressed as
\begin{equation*}
\text{(some polynomial in $\Cal{U}_t(\bar \fu)$)}
+ \sum_{\ga \in \gD^+(\fl)}v^\ga X_{-\ga}
\end{equation*}
with some $v^\ga \in \Cal{U}_t(\bar \fu)$. Therefore the element 
$u \in \Cal{U}_{t+1}(\bar \fu)$ is of the form
\begin{equation*}
u = p + \sum_{\ga \in \gD^+(\fl)}v^\ga X_{-\ga}
\end{equation*}
for some $p, v^\ga \in \Cal{U}_t(\bar \fu)$. By the induction hypothesis, 
the polynomial $p \in \Cal{U}_t(\bar \fu)$ can be then expressed as 
\begin{equation*}
p =AX_{-\vep_i}^2  + \sum_{k=i+1}^nB_k X_{-(\vep_i+\vep_k)}X_{-(\vep_i-\vep_k)}
+ \sum_{\ga \in \gD^+(\fl)}\check{u}^\ga X_{-\ga}
\end{equation*}
for some constants $A$ and $B_k$, and some elements 
$\check{u}^\ga \in \Cal{U}_{t-1}(\bar \fu)$. 
If $u^\ga = \check{u}^\ga + v^\ga$ then $u$ is of the form in 
(\ref{Eqn:Form19}). By induction, this lemma follows.
\end{proof}

\vsp

Now we are ready to show that the map $\varphi_{\Omega_2}$ in (\ref{Eqn:Form15}) 
is standard. Recall that if $1\otimes v^+$ 
is a highest weight vector for $M(-2\vep_i-(n-i-(1/2))\gl_i+\rho)$
with highest weight $-2\vep_i-(n-i-(1/2))\gl_i$ and 
if $\text{pr}:M(-(n-i-(1/2))\gl_i+\rho) \to M_\fq(-(n-i-(1/2))\gl_i+\rho)$ 
is the canonical projection map then $\varphi_{std}(1\otimes v_h) = 
(\text{pr}\circ \varphi)(1\otimes v^+)$, where $\varphi$ is an embedding
of $M(-2\vep_i-(n-i-(1/2))\gl_i+\rho)$ into $M(-(n-i-(1/2))\gl_i+\rho)$.
By Proposition \ref{Prop:Prop1}, we have
$(\text{pr}\circ \varphi)(1\otimes v^+) = \varphi_{std}(1\otimes v_h) \neq 0$.

\begin{Thm}\label{Thm:MapO2Bn(i)}
If $\fq$ is the maximal parabolic 
subalgebra of type $B_n(i)$ for $3 \leq i \leq n-1$ then
the map $\varphi_{\Omega_2}$ %in (\ref{Eqn:Form15}) 
induced by the $\Omega_2|_{V(\mu+\geg)^*}$ system is standard.
\end{Thm}

\begin{proof}
Observe that, as $M(-(n-i-(1/2))\gl_i+\rho) \cong 
\Cal{U}(\bar \fu) \otimes \C_{-(n-i-(1/2))\gl_i}$, 
the vector $\varphi(1\otimes v^+)$ is of the form
$\varphi(1\otimes v^+) = u' \otimes 1_{-(n-i-(1/2))\gl_i}$ 
for some $u' \in \Cal{U}(\bar \fu)$. Since $\varphi(1\otimes v^+)$ 
has weight $-2\vep_i-(n-i-(1/2))\gl_i$, the element $u'$ has weight 
$-2\vep_i$. Thus, by Lemma \ref{Lem:Lem4}, we have
\begin{equation*}
u' =AX_{-\vep_i}^2  + \sum_{k=i+1}^nB_k X_{-(\vep_i+\vep_k)}X_{-(\vep_i-\vep_k)}
+ \sum_{\ga \in \gD^+(\fl)}u^\ga X_{-\ga}
\end{equation*}
for some constants $A$ and $B_k$, and some elements $u^\ga \in \Cal{U}(\bar \fu)$.
As $X_{-\vep_i}$, $X_{-(\vep_i+\vep_k)}$, and $X_{-(\vep_i-\vep_k)}$ are 
not in $\fl$, $\varphi_{std}(1\otimes v_h)$ is given by
\begin{align*}
&\varphi_{std}(1\otimes v_h)
=(\text{pr}\circ \varphi)(1\otimes v^+)\\
&=\text{pr}\bigg(\big(
AX_{-\vep_i}^2  + \sum_{k=i+1}^nB_k X_{-(\vep_i+\vep_k)}X_{-(\vep_i-\vep_k)}
+ \sum_{\ga \in \gD^+(\fl)}u^\ga X_{-\ga}\big) \otimes 1_{-(n-i-(1/2))\gl_i} \big)\bigg)\\
&=\big(AX_{-\vep_i}^2  + \sum_{k=i+1}^nB_k X_{-(\vep_i+\vep_k)}X_{-(\vep_i-\vep_k)}\big)
\otimes 1_{-(n-i-(1/2))\gl_i}.
\end{align*}
Write
$\tu = AX_{-\vep_i}^2  + \sum_{k=i+1}^nB_k X_{-(\vep_i+\vep_k)}X_{-(\vep_i-\vep_k)}$.
Clearly $\tu$ satisfies Condition (H). 
Thus, by Proposition \ref{Prop:Prop2}, there exists a constant 
$c$ so that  $\tu=cu_h$ with $u_h$ in (\ref{Eqn:Form16}). 
By Proposition \ref{Prop:Prop1}, we have $\tu \neq 0$; thus $c\neq 0$.
Since $\varphi_{\Omega_2}(1\otimes v_h) = v_h =u_h \otimes 1_{-(n-i-(1/2))\gl_i}$,
we obtain
\begin{equation*}
\varphi_{\Omega_2}(1\otimes v_h) 
=u_h \otimes 1_{-(n-i-(1/2))\gl_i}
=(1/c)\varphi_{std}(1\otimes v_h).
\qedhere
\end{equation*}
\end{proof}

In Table \ref{THom} below we summarize 
the classification of the maps $\varphi_{\Omega_2}$.

\vskip 0.1in

\begin{table}[h]
\caption{The Homomorphism $\varphi_{\Omega_2}$ for the Non-Heisenberg Case}
\begin{center}
\begin{tabular}{c|c|c}
\hline
Parabolic subalgebra $\fq$ &$\Omega_2|_{V(\mu+\ge_{\gamma})^*}$ & 
$\Omega_2|_{V(\mu+\ge_{n\gamma})^*}$\\
\hline
$B_n(i), 3\leq i \leq n-2$ & standard & non-standard  \\
$B_n(n-1)$ & standard & $?$  \\
$B_n(n)$ & standard & $-$  \\
$C_n(i), 2 \leq i \leq n-1$  &$?$ & standard \\
$D_n(i), 3 \leq i \leq n-3$ & non-standard & non-standard\\
$E_6(3)$ & non-standard & non-standard  \\
$E_6(5)$ & non-standard & non-standard  \\
$E_7(2)$ & non-standard & $-$ \\
$E_7(6)$ & non-standard & non-standard  \\
$E_8(1)$ & non-standard & $-$ \\
$F_4(4)$ & standard & $-$ \\
\hline
\end{tabular}\label{THom}
\end{center}
\end{table}

\appendix

\section{Miscellenious Data }
\label{chap:Data}

In this appendix we recall from \cite{KuboThesis1} 
the miscellenious data for the maximal parabolic subalgebras 
$\fq=\fl \oplus \fg(1) \oplus \fz(\fn)$ of quasi-Heisenberg type shown in
(\ref{Eqn4.0.1}) and (\ref{Eqn4.0.2}) in Section \ref{SS:Prelim}.
For the definition of the deleted Dynkin diagram
see Subsection 4.1 of \cite{KuboThesis1}.

\vsp

%%%%%%%%%%%%%%%%%%%%%%%%%%%%%%%%%%%

\begin{center}
\S $\mathrm{B}_n(i)$, $3\leq i \leq n-2$ 
\end{center}

\begin{enumerate}

\item The deleted Dynkin diagram:
\begin{equation*}
\xymatrix{\belowwnode{\ga_1}\single[r]&
\belowwnode{\ga_2}\single[r]&\dots\single[r]
&\belowwnode{\ga_{i-1}}\single[r]&\belowcnode{\ga_i}\single[r]&
\belowwnode{\ga_{i+1}}\single[r]&\cdots \single[r]
&\belowwnode{\ga_{n-1}}\rdouble[r]&\belowwnode{\ga_n}}
\end{equation*}
%\vspace{3pt}

\item The subgraph for $\flg$:
\begin{equation*}
\xymatrix{\belowwnode{\ga_1}\single[r]&
\belowwnode{\ga_2}\single[r]&
\belowwnode{\ga_3}\single[r]&\dots\single[r]
&\belowwnode{\ga_{i-1}}}
%\single[r]&\belowcnode{\ga_i}\single[r]&
%\belowwnode{\ga_{i+1}}\single[r]&\cdots \single[r]
%&\belowwnode{\ga_{n-1}}\rdouble[r]&\belowwnode{\ga_n}}
\end{equation*}
%\vspace{3pt}

\item The subgraph for $\flng$:
\begin{equation*}
\xymatrix{
\belowwnode{\ga_{i+1}}\single[r]&\cdots \single[r]
&\belowwnode{\ga_{n-1}}\rdouble[r]&\belowwnode{\ga_n}}
\end{equation*}
\vspace{1pt}

\end{enumerate}

We have $\ga_\gamma = \ga_2$.
The highest weight $\mu$ and the set of roots $\gD(\fg(1))$ for 
$\fg(1)$ are $\mu = \vep_1+\vep_{i+1}$ and 
$\gD(\fg(1)) = \{\vep_j \pm \vep_k \; | \; 1 \leq j \leq i \text{ and } i+1 \leq k \leq n\}
 \cup \{\vep_j \; | \; 1\leq j \leq i \}$.
The highest weight $\gamma$ 
and the set of roots $\gD(\fz(\fn))$ for $\fz(\fn)$
are  $\gamma = \vep_1 + \vep_2$
and $\gD(\fz(\fn)) = \{ \vep_j + \vep_k \; | \; 1\leq j<k\leq i\}$.
The highest root $\xig$ and the set of positive roots 
$\gD^+(\flg)$ for $\flg$ are $\xig = \vep_1 - \vep_i$ and 
$\gD^+(\flg) = \{ \vep_j - \vep_k \; | \; 1\leq j < k\leq i\}$.
The highest root $\xing$ and the set of positive roots
$\gD^+(\flng)$ for $\flng$ are $\xing = \vep_{i+1} + \vep_{i+2}$
and $\gD^+(\flng) = \{ \vep_j \pm \vep_k \; | \; i+1 \leq j < k \leq n \}
\cup \{ \vep_j \; | \; i+1 \leq j \leq n\}$.

\vsp

%%%%%%%%%%%%%%%%%%%%%%%%%%%%%%
%\noindent $\bullet \; \mathrm{B}_n(n-1):$ 
\begin{center}
\S $\mathrm{B}_n(n-1)$
\end{center}

\begin{enumerate}

\item The deleted Dynkin diagram:
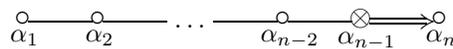
\begin{equation*}
\xymatrix{\belowwnode{\ga_1}\single[r]&
\belowwnode{\ga_2}\single[r]&\dots\single[r]
&\belowwnode{\ga_{n-2}}\single[r]
&\belowcnode{\ga_{n-1}}\rdouble[r]&\belowwnode{\ga_n}}
\end{equation*}
%\vspace{10pt}

\item The subgraph for $\flg$:
\begin{equation*}
\xymatrix{\belowwnode{\ga_1}\single[r]&
\belowwnode{\ga_2}\single[r]&
\belowwnode{\ga_3}\single[r]&\dots\single[r]
&\belowwnode{\ga_{n-2}}}
%\single[r]&\belowcnode{\ga_i}\single[r]&
%\belowwnode{\ga_{i+1}}\single[r]&\cdots \single[r]
%&\belowwnode{\ga_{n-1}}\rdouble[r]&\belowwnode{\ga_n}}
\end{equation*}
%\vspace{10pt}

\item The subgraph for $\flng$:
\begin{equation*}
\xymatrix{\belowwnode{\ga_n}}
\end{equation*}
\vspace{1pt}

\end{enumerate}

We have $\ga_\gamma = \ga_2$.
The highest weight $\mu$ and the set of weights $\gD(\fg(1))$ 
for $\fg(1)$ are $\mu = \vep_1+\vep_{n}$ and 
$\gD(\fg(1)) = \{\vep_j \pm \vep_n \; | \; 1 \leq j \leq n-1 \}
 \cup \{\vep_j \; | \; 1\leq j \leq n-1 \}$.
The highest weight $\gamma$ and the set of weights $\fg(\fz(\fn))$
for $\fz(\fn))$ are  
$\gamma = \vep_1 + \vep_2$ and 
$\gD(\fz(\fn)) = \{ \vep_j + \vep_k \; | \; 1\leq j<k\leq n-1\}$.
The highest root $\xig$ and the set of positive roots 
$\gD^+(\flg)$ for $\flg$ are 
$\xig = \vep_1 - \vep_{n-1}$ and 
$\gD^+(\flg) = \{ \vep_j - \vep_k \; | \; 1\leq j < k\leq n-1\}$.
The highest root $\xing$ and the set of positive roots 
$\gD^+(\flng)$ for $\flng$ are
$\xing = \vep_n$ and 
$\gD^+(\flng) = \{ \vep_n \}$.

\vsp

%%%%%%%%%%%%%%%%%%%%%%%%%%%%%%
%\noindent $\bullet \; \mathrm{B}_n(n) :$ \\

\begin{center}
\S $\mathrm{B}_n(n)$
\end{center}

\begin{enumerate}

\item The deleted Dynkin diagram:
\begin{equation*}
\xymatrix{\belowwnode{\ga_1}\single[r]&
\belowwnode{\ga_2}\single[r]&\dots\single[r]
%&\belowwnode{\ga_{i-1}}\single[r]&\belowcnode{\ga_i}\single[r]&
%\belowwnode{\ga_{i+1}}\single[r]&\cdots \single[r]
&\belowwnode{\ga_{n-1}}\rdouble[r]&\belowcnode{\ga_n}}
\end{equation*}
%\vspace{10pt}

\item The subgraph for $\flg$:
\begin{equation*}
\xymatrix{\belowwnode{\ga_1}\single[r]&
\belowwnode{\ga_2}\single[r]&
\belowwnode{\ga_3}\single[r]&\dots\single[r]
&\belowwnode{\ga_{n-1}}}
%\single[r]&\belowcnode{\ga_i}\single[r]&
%\belowwnode{\ga_{i+1}}\single[r]&\cdots \single[r]
%&\belowwnode{\ga_{n-1}}\rdouble[r]&\belowwnode{\ga_n}}
\end{equation*}

\item No subgraph for $\flng$ ($\flng = \{0 \}$)
\vspace{2pt}

\end{enumerate}

We have $\ga_\gamma = \ga_2$.
The highest weight $\mu$ and the set of weights
$\gD(\fg(1))$ are $\mu = \vep_1$ and 
$\gD(\fg(1)) =  \{\vep_j \; | \; 1\leq j \leq n \}$.
The highest weight $\gamma$ and the set of weights
$\gD(\fz(\fn))$ for $\fz(\fn)$ are 
$\gamma = \vep_1 + \vep_2$ and 
$\gD(\fz(\fn)) = \{ \vep_j + \vep_k \; | \; 1\leq j<k\leq n\}$.
The highest root $\xig$ and the set of positive roots 
for $\flg$ are 
$\xig = \vep_1 - \vep_n$ and 
$\gD^+(\flg) = \{ \vep_j - \vep_k \; | \; 1\leq j < k\leq n\}$.

\vsp

%%%%%%%%%%%%%%%%%%%%%%%%%%%%%%
%\noindent $\bullet \; \mathrm{C}_n(i)$, $2\leq i \leq n-1 :$ \\

\begin{center}
\S $\mathrm{C}_n(i)$, $2\leq i \leq n-1$
\end{center}

\begin{enumerate}

\item The deleted Dynkin diagram:
\begin{equation*}
\xymatrix{\belowwnode{\ga_1}\single[r]&\dots\single[r]
&\belowwnode{\ga_{i-1}}\single[r]&\belowcnode{\ga_i}\single[r]&
\belowwnode{\ga_{i+1}}\single[r]&\cdots \single[r]
&\belowwnode{\ga_{n-1}}\ldouble[r]&\belowwnode{\ga_n}}
\end{equation*}
%\vspace{10pt}

\item The subgraph for $\flg$:
\begin{equation*}
\xymatrix{\belowwnode{\ga_1}\single[r]&
\belowwnode{\ga_2}\single[r]&
\belowwnode{\ga_3}\single[r]&\dots\single[r]
&\belowwnode{\ga_{i-1}}}
%\single[r]&\belowcnode{\ga_i}\single[r]&
%\belowwnode{\ga_{i+1}}\single[r]&\cdots \single[r]
%&\belowwnode{\ga_{n-1}}\rdouble[r]&\belowwnode{\ga_n}}
\end{equation*}
%\vspace{10pt}

\item The subgraph for $\flng$:
\begin{equation*}
\xymatrix{
\belowwnode{\ga_{i+1}}\single[r]&\cdots \single[r]
&\belowwnode{\ga_{n-1}}\ldouble[r]&\belowwnode{\ga_n}}
\end{equation*}
\vspace{1pt}

\end{enumerate}

We have $\ga_\gamma = \ga_1$.
The highest weight $\mu$ and the set of weights 
$\gD(\fg(1))$ for $\fg(1)$ are 
$\mu = \vep_1+\vep_{i+1}$ and 
$\gD(\fg(1)) = \{\vep_j \pm \vep_k \; | \; 1 \leq j \leq i \text{ and } i+1 \leq k \leq n\}$.
The highest weight $\gamma$ and the set of weights $\gD(\fz(\fn))$ 
for $\fz(\fn)$ are 
$\gamma = 2\vep_1$
$\gD(\fz(\fn)) = \{ \vep_j + \vep_k \; | \; 1\leq j<k\leq i\} 
\cup \{2\vep_j \; | \; 1 \leq j \leq i\}$.
The highest root $\xig$ and the set of positive roots 
$\gD^+(\flg)$ for $\flg$ are 
$\xig = \vep_1 - \vep_i$ and 
$\gD^+(\flg) = \{ \vep_j - \vep_k \; | \; 1\leq j < k\leq i\}$
The highest root $\xing$ and the set of positive roots $\gD(\flng)$
for $\flng$ are 
$\xing = 2\vep_{i+1}$ and 
$\gD^+(\flng) = \{ \vep_j \pm \vep_k \; | \; i+1 \leq j < k \leq n \}
\cup \{ 2\vep_j \; | \; i+1 \leq j \leq n\}$.

\vsp

%%%%%%%%%%%%%%%%%%%%%%%%%%%%%%
%\noindent $\bullet \; \mathrm{D}_n(i)$, $3\leq i \leq n-3 :$ \\

\begin{center}
\S $\mathrm{D}_n(i)$, $3\leq i \leq n-3$ 
\end{center}

\begin{enumerate}

\item The deleted Dynkin diagram:
\begin{equation*}
\xymatrix{&&&&&&&&\abovewnode{\ga_{n-1}}\\
\belowwnode{\ga_1}\single[r]&
\belowwnode{\ga_2}\single[r]&\dots\single[r]&\belowwnode{\ga_{i-1}}\single[r]&
\belowcnode{\ga_i}\single[r]&\belowwnode{\ga_{i+1}}\single[r]&\cdots \single[r]&
\wnode\save []+<20pt,0pt>*\txt{$\ga_{n-2}$} \restore\single[ur]\single[dr]& \\
&&&&&&&&\belowwnode{\ga_n} }
\end{equation*}
%\vspace{10pt}

\item The subgraph for $\flg$:
\begin{equation*}
\xymatrix{
\belowwnode{\ga_1}\single[r]&
\belowwnode{\ga_2}\single[r]&
\belowwnode{\ga_3}\single[r]&\dots\single[r]
&\belowwnode{\ga_{i-1}}}
\end{equation*}
%\vspace{10pt}

\item The subgraph for $\flng$:
\begin{equation*}
\xymatrix{&&&\abovewnode{\ga_{n-1}}\\
\belowwnode{\ga_{i+1}}\single[r]&\cdots \single[r]&
\wnode\save []+<20pt,0pt>*\txt{$\ga_{n-2}$} \restore\single[ur]\single[dr]& \\
&&&\belowwnode{\ga_n} }
\end{equation*}
\vspace{1pt}

\end{enumerate}

We have $\ga_\gamma = \ga_2$.
The highest weight $\mu$ and the set of weights $\gD(\fg(1))$
for $\fg(1)$ are 
$\mu = \vep_1+\vep_{i+1}$ and 
$\gD(\fg(1)) = \{\vep_j \pm \vep_k \; | \; 1 \leq j \leq i \text{ and } i+1 \leq k \leq n\}$.
The highest weight $\gamma$ and the set of weights $\gD(\fz(\fn))$ 
for $\fz(\fn))$ are 
 $\gamma = \vep_1 + \vep_2$ and 
$\gD(\fz(\fn)) = \{ \vep_j + \vep_k \; | \; 1\leq j<k\leq i\}$.
The highest root $\xig$ and the set of positive roots $\gD^+(\flg)$
for $\flg$ are 
$\xig = \vep_1 - \vep_i$ and 
$\gD^+(\flg) = \{ \vep_j - \vep_k \; | \; 1\leq j < k\leq i\}$.
The highest root $\xing$ and the set of positive roots $\gD^+(\flng)$
for $\flng$ are 
$\xing = \vep_{i+1} + \vep_{i+2}$
$\gD^+(\flng) = \{ \vep_j \pm \vep_k \; | \; i+1 \leq j < k \leq n \}$.

\vsp

%%%%%%%%%%%%%%%%%%%%%%%%%%%%%%
%\noindent $\bullet \; \mathrm{E}_6(3) :$ 

\begin{center}
\S $\mathrm{E}_6(3)$
\end{center}

\begin{enumerate}

\item The deleted Dynkin diagram:
\begin{equation*}
\xymatrix{
&&\abovewnode{\ga_2}\single[d]&&\\
\belowwnode{\ga_1}\single[r]&\belowcnode{\ga_3}\single[r]&\belowwnode{\ga_4}\single[r]
&\belowwnode{\ga_5}\single[r]&\belowwnode{\ga_6} }
\end{equation*}
%\vspace{10pt}

\item The subgraph for $\flg$:
\begin{equation*}
\xymatrix{
\belowwnode{\ga_2}\single[r]&
\belowwnode{\ga_4}\single[r]&
\belowwnode{\ga_5}\single[r]&
\belowwnode{\ga_{6}}}
\end{equation*}
%\vspace{10pt}

\item The subgraph for $\flng$:
\begin{equation*}
\xymatrix{
\belowwnode{\ga_{1}}}
\end{equation*}
\vspace{1pt}

\end{enumerate}

We have $\ga_\gamma = \ga_2$.
The highest weight $\mu$ for $\fg(1)$ is 
$\mu = \ga_1+ \ga_2 + \ga_3 + 2\ga_4 +2\ga_5 + \ga_6$.
The highest weight $\gamma$ for $\fz(\fn)$
is $\gamma = \ga_1 + 2\ga_2 + 2\ga_3 + 3\ga_4 + 2\ga_5 + \ga_6$.
The highest root $\xig$ for $\flg$ 
is $\xig = \ga_2 + \ga_4 + \ga_5 + \ga_6$.
The highest root $\xing$ for $\flng$
is $\xing = \ga_1$.

\vsp

%%%%%%%%%%%%%%%%%%%%%%%%%%%%%%%%%%%%%%%%%
%\noindent $\bullet \; \mathrm{E}_6(5) :$ 

\begin{center}
\S $\mathrm{E}_6(5)$
\end{center}

\begin{enumerate}

\item The deleted Dynkin diagram:
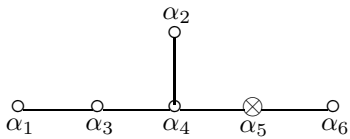
\begin{equation*}
\xymatrix{
&&\abovewnode{\ga_2}\single[d]&&\\
\belowwnode{\ga_1}\single[r]&\belowwnode{\ga_3}\single[r]&\belowwnode{\ga_4}\single[r]
&\belowcnode{\ga_5}\single[r]&\belowwnode{\ga_6} }
\end{equation*}
%\vspace{10pt}

\item The subgraph for $\flg$:
\begin{equation*}
\xymatrix{
\belowwnode{\ga_1}\single[r]&
\belowwnode{\ga_3}\single[r]&
\belowwnode{\ga_4}\single[r]&
\belowwnode{\ga_2}}
\end{equation*}
%\vspace{10pt}

\item The subgraph for $\flng$:
\begin{equation*}
\xymatrix{
\belowwnode{\ga_6}}
\end{equation*}
\vspace{1pt}

\end{enumerate}

We have $\ga_\gamma = \ga_2$.
The highest weight $\mu$ for $\fg(1)$ is 
$\mu = \ga_1+ \ga_2 + 2\ga_3 + 2\ga_4 +\ga_5 + \ga_6$.
The highest weight $\gamma$ for $\fz(\fn)$ is 
$\gamma = \ga_1 + 2\ga_2 + 2\ga_3 + 3\ga_4 + 2\ga_5 + \ga_6$.
The highest weight $\xig$ for $\flg$ is
$\xig = \ga_1 + \ga_2 + \ga_3 + \ga_4$.
The highest weight $\xing$ for $\flng$ is 
$\xing = \ga_6$.

\vsp

%%%%%%%%%%%%%%%%%%%%%%%%%%%%%%%%%%%%%%%%%
%\noindent $\bullet \; \mathrm{E}_7(2) :$ 

\begin{center}
\S $\mathrm{E}_7(2)$
\end{center}

\begin{enumerate}

\item The deleted Dynkin diagram:
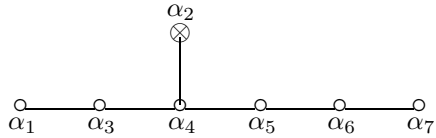
\begin{equation*}
\xymatrix{
&&\abovecnode{\ga_2}\single[d]&&&\\
\belowwnode{\ga_1}\single[r]&\belowwnode{\ga_3}\single[r]&\belowwnode{\ga_4}\single[r]
&\belowwnode{\ga_5}\single[r]&\belowwnode{\ga_6}\single[r]&\belowwnode{\ga_7}}
\end{equation*}
\vspace{2pt}

\item The subgraph for $\flg$:
\begin{equation*}
\xymatrix{
\belowwnode{\ga_1}\single[r]&
\belowwnode{\ga_3}\single[r]&
\belowwnode{\ga_4}\single[r]&
\belowwnode{\ga_5}\single[r]&
\belowwnode{\ga_6}\single[r]&
\belowwnode{\ga_7}}
\end{equation*}
%\vspace{10pt}

\item No subgraph for $\flng$ ($\flng = \{ 0 \}$)
\vspace{2pt}
\end{enumerate}

We have $\ga_\gamma = \ga_1$.
The highest weight $\mu$ for $\fg(1)$ is
$\mu = \ga_1+ \ga_2 + 2\ga_3 + 3\ga_4 +3\ga_5 + 2\ga_6 + \ga_7$.
The highest weight $\gamma$ for $\fz(\fn)$ is
$\gamma = 2\ga_1 + 2\ga_2 + 3\ga_3 + 4\ga_4 + 3\ga_5 + 2\ga_6 + \ga_7$.
The highest root $\xig$  for $\flg$ is 
$\xig = \ga_1 + \ga_3 + \ga_4 + \ga_5+ \ga_6 + \ga_7$.

\vsp

%%%%%%%%%%%%%%%%%%%%%%%%%%%%%%%%%%%%%%%%%
%\noindent $\bullet \; \mathrm{E}_7(6) :$ 

\begin{center}
\S $\mathrm{E}_7(6)$
\end{center}

\begin{enumerate}

\item The deleted Dynkin diagram:
\begin{equation*}
\xymatrix{
&&\abovewnode{\ga_2}\single[d]&&&\\
\belowwnode{\ga_1}\single[r]&\belowwnode{\ga_3}\single[r]&\belowwnode{\ga_4}\single[r]
&\belowwnode{\ga_5}\single[r]&\belowcnode{\ga_6}\single[r]&\belowwnode{\ga_7}}
\end{equation*}
%\vspace{10pt}

\item The subgraph for $\flg$:
\begin{equation*}
\xymatrix{&&&\abovewnode{\ga_{2}}\\
\belowwnode{\ga_1}\single[r]&\belowwnode{\ga_{3}}\single[r]&
\wnode\save []+<20pt,0pt>*\txt{$\ga_{4}$} \restore\single[ur]\single[dr]& \\
&&&\belowwnode{\ga_5} }
\end{equation*}
%\vspace{10pt}

\item The subgraph for $\flng$:
\begin{equation*}
\xymatrix{\belowwnode{\ga_7}}
\end{equation*}
\vspace{1pt}

\end{enumerate}

We have $\ga_\gamma = \ga_1$.
The highest weight $\mu$ for $\fg(1)$ is
$\mu = \ga_1+ 2\ga_2 + 2\ga_3 + 3\ga_4 +2\ga_5 + \ga_6 + \ga_7$.
The highest weight $\gamma$ for $\fz(\fn)$ is 
$\gamma = 2\ga_1 + 2\ga_2 + 3\ga_3 + 4\ga_4 + 3\ga_5 + 2\ga_6 + \ga_7$.
The highest root $\xig$ for $\flg$ is 
$\xig = \ga_1 + \ga_2 + 2\ga_3 + 2\ga_4 + \ga_5$.
The highest root $\xing$ for $\flng$ is 
$\xing = \ga_7$.

\vsp

%%%%%%%%%%%%%%%%%%%%%%%%%%%%%%%%%%%%%%%%%
%\noindent $\bullet \; \mathrm{E}_8(1) :$

\begin{center}
\S $\mathrm{E}_8(1)$
\end{center}

\begin{enumerate}

\item The deleted Dynkin diagram:
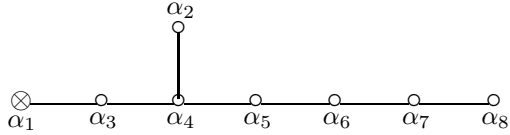
\begin{equation*}
\xymatrix{
&&\abovewnode{\ga_2}\single[d]&&&&\\
\belowcnode{\ga_1}\single[r]&\belowwnode{\ga_3}\single[r]&\belowwnode{\ga_4}\single[r]
&\belowwnode{\ga_5}\single[r]&\belowwnode{\ga_6}\single[r]&\belowwnode{\ga_7}\single[r]&
\belowwnode{\ga_8} \\}
\end{equation*}
%\vspace{10pt}

\item The subgraph for $\flg$:
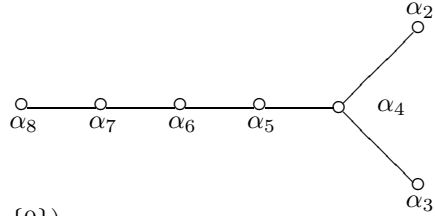
\begin{equation*}
\xymatrix{&&&&&\abovewnode{\ga_{2}}\\
\belowwnode{\ga_8}\single[r]&\belowwnode{\ga_{7}}\single[r]&
\belowwnode{\ga_{6}}\single[r]&\belowwnode{\ga_{5}}\single[r]&
\wnode\save []+<20pt,0pt>*\txt{$\ga_{4}$} \restore\single[ur]\single[dr]& \\
&&&&&\belowwnode{\ga_3} }
\end{equation*}
%\vspace{10pt}

\item No subgraph for $\flng$ ($\flng = \{0\}$)
\vspace{2pt}

\end{enumerate}

We have $\ga_\gamma = \ga_8$.
The highest weight $\mu$ for $\fg(1)$ is 
$\mu = \ga_1+ 3\ga_2 + 3\ga_3 + 5\ga_4 +4\ga_5 + 3\ga_6 + 2\ga_7+\ga_8$.
The highest weight $\gamma$ for $\fz(\fn)$ is 
$\gamma = 2\ga_1 + 3\ga_2 + 4\ga_3 + 6\ga_4 + 5\ga_5 + 4\ga_6 + 3\ga_7+2\ga_8$.
The highest root $\xig$ for $\flg$ is 
$\xig = \ga_2+\ga_3+2\ga_4+2\ga_5+2\ga_6+2\ga_7+\ga_8$.

\vsp

%%%%%%%%%%%%%%%%%%%%%%%%%%%%%%%%%%%%%%%%%
%\noindent $\bullet \; \mathrm{F}_4(4) :$

\begin{center}
\S $\mathrm{F}_4(4)$
\end{center}

\begin{enumerate}

\item The deleted Dynkin diagram:
\begin{equation*}
\xymatrix{
\belowwnode{\ga_1}\single[r]&\belowwnode{\ga_2}\rdouble[r]&
\belowwnode{\ga_3}\single[r]&\belowcnode{\ga_4}}
\end{equation*}
%\vspace{10pt}

\item The subgraph for $\flg$:
\begin{equation*}
\xymatrix{
\belowwnode{\ga_1}\single[r]&\belowwnode{\ga_2}\rdouble[r]&
\belowwnode{\ga_3}}
\end{equation*}
%\vspace{10pt}

\item No subgraph for $\flng$ ($\flng = \{0\}$)
\vspace{2pt}

\end{enumerate}

We have $\ga_\gamma = \ga_1$.
The highest weight $\mu$ for $\fg(1)$ is 
$\mu = \ga_1+ 2\ga_2 + 3\ga_3 + \ga_4$.
The highest weight $\gamma$ for $\fz(\fn)$ is 
$\gamma = 2\ga_1 + 3\ga_2 + 4\ga_3 + 2\ga_4$.
The highest root for $\xig$ for $\flg$ is
$\xig = \ga_1 + 2\ga_2 + 2\ga_3$.

\bibliography{main}
\bibliographystyle{amsplain}

\end{document}